\documentclass[10pt]{article}
\usepackage{hyperref}
\usepackage{amssymb}
\usepackage{amsmath}
\usepackage{enumerate}
\usepackage[shortlabels]{enumitem}
\newcommand{\bone}{\mathbf{1}}
\input{liemacs10.sty} 
\usepackage[all]{xy}

\numberwithin{equation}{section}

\usepackage{xcolor}

\def\hypothesisname{Hypothesis}
\newtheorem{hypo}[theor]{\hypothesisname}

\newcommand\be{{\bf{e}}}

\newcommand{\ee}{\mathrm{e}}
\newcommand{\ie}{\mathrm{i}}

\newcommand{\sF}{{\sf F}}
\newcommand{\sH}{{\sf H}}

\newcommand{\sV}{{\tt V}}
\newcommand{\sE}{{\tt E}}

\renewcommand{\phi}{\varphi}

\renewcommand\mlabel{\label}

\begin{document}

\title{Crowned Lie groups and nets of real subspaces}
\author{Daniel Belti\c t\u a\footnote{Institute of Mathematics ``Simion Stoilow'' of the Romanian Academy,
		P.O. Box 1-764, Bucharest, Romania. 
		\texttt{Email}: Daniel.Beltita@imar.ro, beltita@gmail.com} 
	\quad and \quad 
	Karl-Hermann Neeb\footnote{Department Mathematik, Friedrich-Alexander-Universit\"at, Erlangen-N\"urnberg,
		Cauerstrasse 11, 91058 Erlangen, Germany. 
		\texttt{Email}: neeb@math.fau.de}}
	\date{June 19, 2025}

\maketitle

\begin{abstract}
  We introduce the notion of a complex crown domain  for a connected Lie
  group $G$, and we use analytic extensions of orbit maps of antiunitary
representations to these domains to construct nets of real subspaces
on $G$ that are isotone, covariant and satisfy the Reeh--Schlieder
and Bisognano--Wichmann conditions from Algebraic Quantum Field Theory.
This provides a unifying perspective on various constructions of such nets. 
The representation theoretic properties of different crowns
are discussed in some detail for the non-abelian $2$-dimensional
Lie group $\Aff(\R)$. We also characterize the existence of nets
with the above properties by a regularity condition
in terms of an Euler element in the Lie algebra $\g$ and show that
all antiunitary representations of the split oscillator group
have this property. 
\end{abstract}

\setcounter{tocdepth}{1}
\tableofcontents 

\vspace{8mm}

\section*{Introduction}

Nets of real subspaces of a complex Hilbert space 
occur naturally in Algebraic Quantum Field Theory (AQFT)
in connection with nets of algebras of local observables
(\cite{MN21, MN24, FNO25a, NO21}) 
and play a central role in new constructions in AQFT (\cite{MMTS21, LL15}). 
The main motivation for this article was to find minimal
structures that allow to construct well-behaved nets of real subspaces 
from antiunitary representations of Lie groups,
that are not necessarily reductive, and to develop a unifying
perspective on the constructions of such nets in
\cite{NO21, NOO21, MN24, FNO25a}. 
In this context, standard subspaces play a key role: 
A closed real subspace $\sV$ of a complex Hilbert space
$\cH$ is called a {\it standard subspace} if
$\sV + \ie \sV$ is dense in $\cH$ 
 ($\sV$ is {\it cyclic}) 
 and $\sV \cap \ie \sV = \{0\}$ ($\sV$ is {\it separating}).
 Then  the complex conjugation $T_\sV$ on $\sV + \ie \sV$
 (the Tomita involution) is closed and densely defined in $\cH$,
 hence has a polar decomposition 
$T_\sV = J_\sV \Delta_\sV^{1/2}$, where
$J_\sV$ is a conjugation (an antilinear involutive isometry)
and $\Delta_\sV$ is a positive selfadjoint operator
satisfying $J_\sV \Delta_\sV J_\sV = \Delta_\sV^{-1}$.
The unitary one-parameter group $U_t := \Delta_\sV^{\ie t}$ 
(the {\it modular group of $\sV$}) preserves the subspace $\sV$
and commutes with $J_\sV$. 
We refer to \cite{Lo08, NO17} for more on standard subspaces
and antiunitary representations in general.

If $\tau$ is an involutive automorphism of a Lie 
group $G$ and $G_\tau := G \rtimes_{\tau} \{\pm 1 \}$, then an
{\it antiunitary representation $(U,\cH)$ of $G_\tau$}
is a homomorphism $U \: G_\tau \to \AU(\cH)$
into the group of unitary or antiunitary operators on $\cH$, 
for which $U\res_G$ is a (continuous) unitary representation
and $U(\tau)$ is a {\it conjugation}. The preceding discussion
shows that antiunitary representations
of $\R_{\id} \cong \R \times \{\pm 1 \} \cong \R^\times$ are in one-to-one
correspondence with standard subspaces. More generally,
if $(U,\cH)$ is an antiunitary representation
of $G_\tau$ and $h \in \g$ (the Lie algebra of $G$) is fixed by 
$\tau$, then we obtain a standard subspace
\begin{equation}
  \label{eq:vhu-intro}
 \sV(h,U) := \Fix(J \Delta^{1/2}) \quad
  \mbox{ for } \quad J = U(\tau)
  \quad \mbox{ and }\quad
 \Delta :=  \ee^{2\pi \ie \partial U(h)}. 
\end{equation}
where $\partial U(h) = \derat0  U(\exp th)$ is the skew-adjoint
infinitesimal generator of the unitary one-parameter group $U_h(t) :=
U(\exp th)$. This standard subspace has another description
in terms of a condition resembling the KMS (Kubo--Martin--Schwinger)
condition for states of operator algebras (cf.\ \cite{BR96}). 
The standard subspace $\sV(h,U)$ coincides with
the set of all vectors $\xi \in \cH$
for which the orbit map $U^\xi_h(t) := U_h(t)\xi$ 
extends analytically 
to a continuous map $U^\xi \: \oline{\cS_\pi} \to \cH$ on
the closure of the strip $\cS_\pi = \{  z \in \C \: 0 < \Im z < \pi\}$,
satisfying
$U^\xi_h(\pi \ie) = J \xi$. Then $U^\xi(\pi \ie/2) \in \cH^J$ is a
$J$-fixed vector whose orbit map extends to the strip
\begin{equation}
  \label{eq:spi}
  \cS_{\pm \pi/2} := \{ z \in \C \: |\Im z| < \pi/2\} \subeq \C
\end{equation}
(\cite[Prop.~2.1]{NOO21}). 

To extend this one-dimensional picture to
higher dimensional Lie groups $G$, one has to specify complex
manifolds $\Xi$ (crown domains), containing $G$ as a totally real submanifold,
to which orbit maps of $J$-fixed analytic vectors extend.
These domains $\Xi$ generalize the strip 
$\cS_{\pm \pi/2}$, corresponding to the one-dimensional Lie group $G = \R$.
For semisimple Lie groups complex crown domains, such as
  those obtained from crowns of Riemannian symmetric spaces $G/K$ 
of non-compact type  as their inverse images
$\Xi_{G_\C}$ in the complexified group $G_\C$,  
are particular well-known examples that have been used in
harmonic analysis (cf.\ \cite{AG90}, \cite{KSt04}, \cite{FNO25a}).
But they have never  been studied systematically from the
general perspective we introduce here. 
We call $\Xi$ a {\it crown domain of $G$}. 

To understand natural requirements for crown domains
that are needed for the constructions of local nets,
we take a closer look at how
standard subspaces and nets of real subspaces of a complex Hilbert space 
occur in AQFT. 
A family $\sH(\cO)$ of closed real subspaces of $\cH$, indexed by open subsets
$\cO$ of a $G$-space~$M$, is called a {\it net of real subspaces on $M$}.
For such a net we consider the following properties
with respect to an antiunitary representation $(U,\cH)$ of $G_\tau$: 
\begin{itemize}
\item[(Iso)] {\bf Isotony:} $\cO_1 \subeq \cO_2$ 
implies $\sH(\cO_1) \subeq \sH(\cO_2)$ 
\item[(Cov)] {\bf Covariance:} $U(g)\sH(\cO) = \sH(g\cO)$ for $g \in G$. 
\item[(RS)] {\bf Reeh--Schlieder property:} 
$\sH(\cO)$ is cyclic  if $\cO \not=\eset$. 
\item[(BW)] {\bf Bisognano--Wichmann property:}
There exist $h \in \g^\tau$ and an open connected
subset $W \subeq M$ (invariant under $\exp(\R h)$), 
such that $\sH(W)$ is the standard subspace $\sV = \sV(h,U)$
specified by  $\Delta_\sV = \ee^{2\pi \ie \partial U(h)}$ and $J_\sV = U(\tau)$
as in \eqref{eq:vhu-intro}. 
\end{itemize}

The key idea for constructing such nets of real subspaces
is to extend the description of $\sV(h,U)$ in
terms of a KMS condition on the complex strip to general Lie groups.
This leads us to natural requirements on $\Xi$ and~$h$ described below.
From \cite[Thm.~3.1]{MN24} we know that a necessary condition
for the existence of non-trivial nets satisfying (Iso), (Cov), (RS) and
(BW) is that $h \in \g$ is an {\it Euler element}, i.e., $0 \not=\ad h$
is diagonalizable with its eigenvalues contained in $\{1,0,-1\}$.
Then its eigenspaces $\g_\lambda = \g_\lambda(h)$ define a $3$-grading
of $\g$, where $\g_1$ or $\g_{-1}$ may be trivial. 
We therefore assume from the outset that $h \in \g$ is an
Euler element. Then $\tau_h^\g := \ee^{\pi \ie \ad h}$ is an involutive
automorphism of $\g$, 
mapping
$x = x_1 + x_0 + x_{-1}$ with $x_j \in \g_j$ to
$-x_1 + x_0 - x_{-1}$, and we assume that it integrates to an
involutive automorphism $\tau_h$ on $G$, so that we can form
the group
$G_{\tau_h} = G \rtimes \{\bone,\tau_h\} :=  G \rtimes_{\tau_h} \{\pm 1\}$
  (cf.\ \eqref{eq:gtauh} below). 
  Our requirements on
the complex manifold $\Xi$ are: 
\begin{itemize}
\item[\rm(Cr1)] The natural action of $G_{\tau_h}$ on $G$ by
  $(g,1).x = g x$ and $(e,-1).x = \tau_h(x)$
    for $g,x \in G$   extends to an
  action on $\Xi$, such that $G$ acts by holomorphic maps  
  and $\tau_h$ by an antiholomorphic involution, denoted $\oline\tau_h$. 
\item[\rm(Cr2)] There exists a connected open $e$-neighborhood
  $W^c \subeq \Xi^{\oline\tau_h}$
  (the set of $\oline\tau_h$-fixed points), such that, for every $p \in W^c$, the orbit map
  \[ \alpha^p \: \R \to \Xi, \quad  \alpha^p(t) := \exp(th).p \]
  extends to a holomorphic map
  $\cS_{\pm \pi/2} \to \Xi$. 
\item[\rm(Cr3)] The map $\eta_G \: G \to G_\C$ 
  extends to a holomorphic ($G_{\tau_h}$-equivariant) 
  map $\eta_\Xi \: \Xi \to G_\C$ which is a covering of the open subset
  $\Xi_{G_\C} := \eta_\Xi(\Xi)$. 
\end{itemize}

We combine this data  in the triple $(G,h,\Xi)$, 
which we call a {\it crowned Lie group}. Crowned Lie groups
form a category in the obvious way.  
For a crowned Lie group $(G,h,\Xi)$
and an antiunitary
representation $(U,\cH)$ of $G_{\tau_h}$, we write
\[ \cH^\omega(\Xi) \subeq \cH \]
for the subspace of those analytic vectors for $G$ whose orbit map extends
to $\Xi$. The non-triviality of this space
imposes restrictions on $\Xi$.
This follows already from the discussion in the
last section of \cite{BN24}, where we have seen 
that, for the group $G = \Aff(\R)_e \cong \R \rtimes \R$, the domain
$\Xi$ must be contained in $\C \times \cS_{\pm\pi/2}$ 
($\subeq \C\rtimes\C=G_\C$) for our constructions of
nets to work (cf.\ Theorem~\ref{thm:4.9}). 
To specify the necessary boundary behavior of the 
extended orbit maps on $\Xi$, we write 
\begin{equation}
  \label{eq:hjtemp}
  \cH^J_{\rm temp} \subeq \cH^J
\end{equation}
for the dense real linear subspace of the $J$-fixed vectors
$v$, for which the orbit map $U^v_h(t) = U(\exp th)v$
extends to a holomorphic map $\cS_{\pm\pi/2} \to \cH$,  and the limit
\[ \beta^+(v) := \lim_{t \to -\pi/2} U^v_h(\ie t) \]
exists in the subspace $\cH^{-\infty}_{U_h}$ 
of distribution vectors of the one-parameter group~$U_h$
in the weak-$*$ topology 
(see Appendix~\ref{app:c}). It always exists in the larger space
$\cH^{-\omega}_{U_h}$ of hyperfunction vectors, so that the main
point is the addition regularity (temperedness), related to the
weak convergence after pairing with elements of~$\cH^\infty_{U_h}$.
For any real linear subspace
\begin{equation}
  \label{eq:fincl}
  \sF \subeq \cH^\omega(\Xi) \cap \cH^J_{\rm temp}, 
\end{equation}
we then obtain a real subspace
\begin{equation}
  \label{eq:def3}
  \sE := \beta^+(\sF) \subeq \cH^{-\infty}_{U_h} \subeq \cH^{-\infty},
\end{equation}
and from this space we construct a net of real subspaces indexed
by open subsets of $G$ as follows.
To an open subset $\cO \subeq G$, we associate the closed real subspace 
\begin{equation}
  \label{eq:HE}
  \sH_\sE^G(\cO) := \oline{\spann_\R U^{-\infty}(C^\infty_c(\cO,\R))\sE}.
\end{equation}
The operators $U^{-\infty}(\phi)$, $\phi \in C_c^\infty(G,\R)$,  map 
$\cH^{-\infty}$ into $\cH$ because they are adjoints of
continuous operators $U(\phi) = \int_G \phi(g) U(g)\, dg
\colon \cH \to \cH^\infty$
  (Appendix~\ref{app:c}).
Accordingly, the closure in \eqref{eq:HE} is taken with respect
to the topology of $\cH$. 
  The net $\sH^G_\sE$ trivially satisfies (Iso) and (Cov),
  and our first main result (Theorem~\ref{thm:4.9}) asserts that,
  if $\sF$ is $G$-cyclic in the sense that $U(G)\sF$
    spans a dense subspace, 
  for $\sE$ as in \eqref{eq:def3},
  the net $\sH^G_\sE$ on $G$ also satisfies (RS) and (BW). 

  To pass from a net on $G$ to a net on
  a homogeneous space $M = G/H$ is easy: With the projection map
$q \colon G \to M$, we obtain a ``pushforward net''
\begin{equation}
  \label{eq:pushforward}
  \sH^M_\sE(\cO) := \sH_\sE^G(q^{-1}(\cO)).
\end{equation}
It automatically satisfies (Iso),  (Cov) and (RS),
but not always (BW). 
If $\sE$ is invariant under $U^{-\infty}(H)$, then
$\sH^G_\sE(\cO)= \sH^G_\sE(\cO H)$ for any open subset
$\cO \subeq G$ (\cite[Lemma~2.11]{NO21}), and this implies that 
(BW) holds for $W^M := q(W^G) \subeq G/H$.
So the crucial point
is to find $\sE$ in such a way that it is $H$-invariant. 
In \cite{FNO25a} this was done for semisimple groups
and non-compactly causal symmetric spaces~$G/H$.
Here the Kr\"otz--Stanton Extension Theorem \cite{KSt04} 
  and Simon's generalization \cite{Si24}
  for non-linear groups were used to verify the existence of the
  limits $\beta^+(v)$. 
We plan to explore nets on homogeneous space in a subsequent project.
Here we focus on the group case. 

  Whenever a net $(\sH(\cO))_{\cO \subeq G}$ of real subspaces
  exists for an antiunitary representation $(U,\cH)$ of $G_{\tau_h}$, 
  such that (Iso), (Cov), (RS) and (BW) are satisfied, then the
  representation $U$ is {\it $h$-regular} in the sense that
there exists an identity neighborhood $N \subeq G$
  such  that, for the associated standard subspace
  $\sV := \sV(h,U)$ as in (BW), the intersection $\sV_N
  := \bigcap_{g \in N} U(g) \sV$ is still  standard
  (\cite[\S 4]{MN24}).
  We conjecture that this property is satisfied for {\bf all}
  antiunitary representations of $G_{\tau_h}$.
  This conjecture has been confirmed for reductive 
  groups in \cite[Thm.~4.23]{MN24}, and in \cite{MN24}
  it is proved for a
  variety of other classes of Lie groups under suitable requirements.
  In Theorem~\ref{thm:reg-net} we show that $h$-regularity
  is actually equivalent to the existence of an open subset $W \subeq G$,
  for which a net satisfying (Iso), (Cov), (RS) and (BW) exists.
  The smallest pair $(G,h)$ for which all  the $h$-regularity
  criteria from \cite{MN24} fail, is the split 
  oscillator group  of dimension~$4$.
  This example   is discussed in Section~\ref{sec:5},
  where we show in particular that the regularity conjecture also
  holds for this group and all its Euler elements.\\

  The {\bf structure of this paper} is as follows. 
  In Section~\ref{sec:1} we discuss our axioms for crowned Lie groups
  and describe some natural constructions of these structures,
  including the semisimple crown domains and complex Olshanski semigroups.
  Section~\ref{sec:2} is devoted to the construction of nets of
  real subspaces satisfying (Iso), (Cov), (RS) and (BW)
  from antiunitary representations $(U,\cH)$, for which the space
  $\cH^\omega(\Xi) \cap \cH^J_{\rm temp}$ is $G$-cyclic
  (Theorem~\ref{thm:4.9}). Major classes of groups to which
  our setting applies are semisimple Lie groups and
  Lie groups whose Lie algebra contains a pointed generating invariant
  cone $C$, so that we can form complex Olshanski semigroups
  $G \exp(\ie C)$. This provides a unification of the analytic
  extension to complex semigroups that
  requires positive spectrum conditions but lives on rather large domains 
 (\cite{Ne00})   and the analytic extension arguments
 related to crown domains of symmetric spaces applying to larger
 classes of representations (\cite{KSt04}). 
  We explain in Section~\ref{sec:2} how these two classes
  connect to our setting and also add a brief subsection on 
  the Poincar\'e group, acting on Minkowski space. 
  The smallest non-abelian example, the affine group of the real line, 
  is discussed in Section~\ref{sec:3}. We show that there is a natural 
  crown $\Xi_1$ of $G$ that is too large because 
  $\cH^\omega(\Xi_1) \cap \cH^J_{\rm temp}$ vanishes for infinite-dimensional
  irreducible representations, but there exists a smaller crown domain 
  $\Xi_2$ for which the assumptions of Theorem~\ref{thm:4.9} are satisfied
  for all antiunitary representations.

In Section~\ref{sec:4} $h$-regularity 
 is shown to be equivalent to the existence of nice nets
  (Theorem~\ref{thm:reg-net}), 
  and in Section~\ref{sec:5} we show that all representations
  of the split oscillator group are $h$-regular, which verifies our
  general conjecture for this group. We conclude with a short
  section on perspectives and open problems. \\

\nin {\bf Notation:}
\begin{itemize}
\item We write $e \in G$ for the identity element in the Lie group~$G$ 
and $G_e$ for its identity component. 
\item For $x \in \g$, we write $G^x := \{ g \in G \: \Ad(g)x = x \}$ 
for the stabilizer of $x$ in the adjoint representation 
and $G^x_e = (G^x)_e$ for its identity component. 
\item For $r > 0$ we denote the corresponding horizontal strips in $\C$
  by
  \[ \cS_r := \{ z \in \C \colon 0 < \Im z < r\} \quad \mbox{ and } \quad 
    \cS_{\pm r} := \{ z \in \C \colon |\Im z | < r\}.\]
\item For a unitary representation $(U,\cH)$ of $G$ we write:
  \begin{itemize}
  \item[\rm(a)] $\partial U(x) = \derat0 U(\exp tx)$ for the infinitesimal
    generator of the unitary one-parameter group $(U(\exp tx))_{t \in\R}$
    in the sense of Stone's Theorem. 
  \item[\rm(b)] $\dd U \colon \cU(\g_\C) \to \End(\cH^\infty)$ for the representation of
    the enveloping algebra $\cU(\g_\C)$ of $\g_\C$ on
    the space $\cH^\infty$ of smooth vectors. Then
    $\partial U(x) = \oline{\dd U(x)}$ for $x \in \g$. 
  \end{itemize}
\end{itemize}

\nin {\bf Acknowledgment:} 
Part of this work has been carried out during a research visit of the
authors at the Erwin-Schr\"odinger-Institut in Vienna during
the Thematic Programme ``Infinite-dimensional Geometry:
Theory and Applications''.

DB acknowledges partial financial support from 
the Research Grant GAR 2023 (code 114), supported from the Donors' Recurrent Fund of the Romanian Academy, managed by the``PA\-TRI\-MO\-NIU'' Foundation. 

KHN also acknowledges support of the Institut Henri Poincar\'e
(UAR 839 CNRS-Sorbonne Universit\'e), and LabEx CARMIN (ANR-10-LABX-59-01).

We are most grateful to Jonas Schober for providing the example discussed
in Subsection~\ref{subsec:schober}.
We are also indebted to Michael Preeg and Tobias Simon for comments
on a first draft of this paper. 

\section{Crowned Lie groups}
\mlabel{sec:1} 

We consider the following setting:
\begin{itemize}
\item $G$ is a connected Lie group whose universal complexification
  $\eta_G \: G \to G_\C$ has discrete kernel.
  If $G$ is simply connected, then $G_\C$ is the simply connected
  group with Lie algebra $\g_\C$, and this condition is satisfied
  (cf.\ \cite[Thm.~15.1.4(i)]{HiNe12}). 
\item $h \in \g$ is an Euler element, 
i.e., $(\ad h)^3=\ad h \not=0$, 
for which the associated involution $\tau_h^\g = \ee^{\pi \ie \ad h}$
of $\g$ integrates to an involutive automorphism
$\tau_h$ of $G$. This is always the case if  $G$ is simply connected.
We write
  \begin{equation}
    \label{eq:gtauh}
G_{\tau_h} = G \rtimes \{\bone,\tau_h\} :=  G \rtimes_{\tau_h} \{\pm 1\}
\end{equation}
for the corresponding semidirect product
and abbreviate $\tau_h = (e,-1)$ for the corresponding element of
  $G_{\tau_h}$. 
The universality of $G_\C$ implies the existence of a unique antiholomorphic involution $\oline\tau_h$ on $G_\C$ satisfying $\oline\tau_h \circ \eta_G = \eta_G \circ \tau_h$.
\end{itemize}

\subsection{Axioms for crown domains of Lie groups} 
\mlabel{sec:crownlie}

We now present an axiomatic specification of crown domains
of $G$ to which orbit maps of $J$-fixed vectors in antiunitary
representations may extend in such a way that boundary values
lead to nets of real subspaces on $G$ as in \eqref{eq:HE}. 

\begin{defn} \mlabel{def:1.1}
  (a) A {\it $(G,h)$-crown domain} is a connected
  complex manifold $\Xi$ containing
$G$ as a closed totally real submanifold, such that the following conditions
are satisfied:
\begin{itemize}
\item[\rm(Cr1)] 
The natural action of $G_{\tau_h}$ on $G$
by  $(g,1).x = gx$ and $(e,-1).x = \tau_h(x)$ for $g,x \in G$, 
  extends to an
  action on $\Xi$, such that $G$ acts by holomorphic maps  
  and $\tau_h$ by an antiholomorphic involution, 
  again denoted $\oline\tau_h$. 
  These extensions are unique because $G$ is totally real in $\Xi$.
\item[\rm(Cr2)] 
There exists an $e$-neighborhood 
$W^c$ in $\Xi^{\oline\tau_h}$  
(the set of $\oline\tau_h$-fixed points) such that, for every $p \in W^c$, 
the orbit map
  \[ \alpha^p \: \R \to \Xi, \quad  \alpha^p(t) := \exp(th).p \]
  extends  to a holomorphic map
  $\cS_{\pm \pi/2} \to \Xi$. 
\item[\rm(Cr3)] 
The map $\eta_G \: G \to G_\C$ 
  extends to a holomorphic ($G_{\tau_h}$-equivariant) 
  map $\eta_\Xi \: \Xi \to G_\C$ which is a covering of the open subset
  $\Xi_{G_\C} := \eta_\Xi(\Xi)$ 
so we have the commutative diagram 
  \[\xymatrix{
  G\ \ar@{^{(}->}[r] \ar[dr]_{\eta_G} & \Xi \ar[d]^{\eta_\Xi} \\
 & \Xi_{G_\C}  }\]
\end{itemize} 

\nin (b) We call the 
triple $(G,h,\Xi)$   a {\it crowned Lie group}. 
  For crowned Lie groups 
$(G_j, h_j, \Xi_j)$, 
  $j = 1, 2$, a
  {\it morphism of crowned Lie groups}
  is a holomorphic map $\phi \: \Xi_1 \to \Xi_2$,
  restricting to a Lie group morphism
  $\phi_G \: G_1 \to G_2$ such that
  \[\L(\phi_G) h_1 = h_2    .\]
  This implies that $\phi_G \circ \tau_{h_1} = \tau_{h_2} \circ \phi_G$,
  and, by analytic continuation, 
  $\phi$ intertwines the $G_{1,\tau_{h_1}}$-action
  on $\Xi_1$ with the $G_{2,\tau_{h_2}}$-action on $\Xi_2$.
\end{defn}

\begin{rem}
	\mlabel{remW}
	 (On  the condition (Cr2)) 
Let us denote by $\cW$ the set of all points $p\in \Xi^{\oline\tau_h}$ 
whose orbit map $\alpha^p$ has the holomorphic extension property 
referred to in the axiom (Cr2). 
 The fixed point set $\Xi^{\oline\tau_h}$ is invariant under the
  connected subgroup $G^h_e$ which commutes with the involution
  $\oline\tau_h$ on $\Xi$. As it also commutes with $\exp(\R h)$, 
 the set $\cW$ is also $G^h_e$-invariant. 
  This shows that, if $e$ belongs to the interior of $\cW$, then 
  there exists a $G^h_e$-invariant connected open subset $W^c\subseteq\cW$ with $e\in W^c$.

  If $\Omega' \subeq \g^{-\tau_h^\g}$ is a convex open $0$-neighborhood 
  with $\exp(\ie\Omega') \subeq \eta_\Xi(\Xi)$, then the exponential
  function $\ie\Omega' \to \eta_\Xi(\Xi)$ lifts uniquely to an analytic map
  \begin{equation}
    \label{eq:exp-omega'}
 \exp \: \ie\Omega' \to \Xi \quad \mbox{ with } \quad 
 \exp(0) = e, \ \mbox{ the unit element in } G \subeq \Xi.
  \end{equation}
 Since $\ie\g^{-\tau_h^\g}\subseteq\g_\C^{\oline\tau_h^\g}$, we have 
  $\exp(\ie\Omega') \subeq \Xi^{\oline\tau_h}$, so that
  any $G^h_e$-invariant domain $W^c$ contains an open subset of the form
  $G^h_e.\exp(\ie\Omega')$. 
  For this reason, we may assume that $W^c$ is of this form. 

  For any element $x = x_1 + x_{-1} \in \g^{-\tau_h^\g}$ with
  $x_{\pm 1} \in \g_{\pm 1}(h)$, we have
  \[ \zeta(\ie x) = x_1 - x_{-1} \in \g^{-\tau_h^\g}\quad \mbox{ for } \quad
    \zeta := e^{-\frac{\pi \ie}{2} \ad h} \in \Aut(\g_\C).\]
  We thus associate to $W^c = G^h_e.\exp(\ie \Omega')$, the open subset
  \begin{equation}
    \label{eq:defwg}
 W^G := G^h_e.\exp(\Omega)  \subeq G\quad \mbox{ with }\quad
 \Omega := \ie \zeta(\Omega') \subeq \g^{-\tau_h^\g}.
  \end{equation}
  \end{rem}

  \begin{rem} (On the existence of crown domains)
    If $\eta_G$ is injective, we may identify $G$ with a closed subgroup
  of $G_\C$. In this case $\Xi$ is an open subset of $G_\C$, invariant
  under the $G_{\tau_h}$-action and 
  (Cr3) 
  is trivially satisfied. 
  Typical domains with this property are easily constructed
  as products $\Xi := G \exp(\ie \Omega)$, where $\Omega \subeq \g$
  is an open convex 
  $0$-neighborhood invariant under $-\tau_h^\g$, for which the
  polar map $G \times \Omega \to G_\C, (g,x) \mapsto g \exp x$ is a
  diffeomorphism onto an open subset. 
  Then $\Xi^{\oline\tau_h}
  = G^{\tau_h}.\exp(\ie \Omega^{-\tau_h^\g}) \supeq G^h_e.\exp(\ie \Omega^{-\tau_h^\g})$,
  so that any sufficiently small open convex $0$-neighborhood
  $\Omega' \subeq \Omega^{-\tau_h^\g}$ specifies 
  an open neighborhood
  $W^c := G^h_e.\exp(\ie\Omega')$ of $e\in \Xi^{\oline\tau_h}$,  
  hence   (Cr2) is satisfied.
  This condition follows from 
  $\exp(\ie t h) \exp(\ie\Omega') \subeq \Xi$ for $|t| \leq \pi/2$
and $\Omega'$ small enough. 
  Therefore crown domains $\Xi$ satisfying 
(Cr1--3) 
  exist in abundance. 
\end{rem}

\begin{ex} For $G = \R\subeq \C = G_\C$
  and $h = 1$ (a basis element in $\g = \R$),
any strip  
  \[ \Xi = \cS_{\pm r} = \{ z \in \C \: |\Im z| < r \} \subeq \C = G_\C, \qquad
    r  \geq \pi/2, \]
  is a crown domain for $G = \R$ with $W^c = G$.
  In this case $\tau_h = \id$, $\oline \tau_h(z) = \oline z$
    and $G_{\tau_h} \cong (\R^\times, \cdot)$.  
\end{ex}

Below we shall encounter various kinds of non-abelian examples,
see in particular Examples~\ref{ex:affine-group} and~\ref{ex:olshanski}. 

\subsection{Constructions of crown domains}

The following lemma is a tool for constructing crown domains
that we shall use in various contexts.

\begin{lem} \mlabel{lem:Mxi} Suppose that $\eta_G$ is injective
  and that $G_{\C, \oline\tau_h}$ acts on the complex manifold $M$
  in such a way that $G_\C$ acts by holomorphic maps and
  $\oline\tau_h$  by an antiholomorphic map $\tau_{M}$.
Let $\Xi_{M} \subeq M$ be a $G_{\tau_h}$-invariant connected open subset  
   for which
  there exists an open subset $W^{M,c}\subeq \Xi_M^{\tau_M}$ satisfying 
  \[ W^{M,c} \subeq \{ m \in \Xi_M^{\tau_M} \:
    \exp(\cS_{\pm \pi/2}h).m \subeq \Xi_M \}. \]
Then, for $m_0 \in W^{M,c}$,  the open subset 
  \[ \Xi := \{ g \in G_\C \: g.m_0 \in \Xi_{M}\}\]
  satisfies {\rm(Cr1-3)} with the open subset 
$W^c := \{ g \in G_\C^{\oline\tau_h}  \: g.m_0 \in W^{M,c}\}
  \subseteq\Xi^{\oline\tau_h}.$ 
\end{lem}

\begin{prf} (Cr1): Since $\oline\tau_h(m_0) = m_0$, $\oline\tau_h(\Xi) = \Xi$
  follows from the antiholomorphic action of $G_{\tau_h}$ on~$\Xi_{M}$.

  \nin (Cr2): The inclusion $\exp(\cS_{\pm \pi/2} h) W^c \subeq \Xi$
  follows from $\exp(\cS_{\pm \pi/2} h) W^{M,c} \subeq \Xi_{M}$.

  \nin (Cr3) is redundant because $G \subeq G_\C$. 
\end{prf}

\begin{ex} (The affine group of the line) \mlabel{ex:affine-group} 
  We consider the $2$-dimensional affine group of the real line 
$G = \Aff(\R)_e \cong \R \rtimes \R_+$ with
$\g = \R  x \rtimes \R h$, $x = (1,0), h = (0,1)$, so that 
$[h,x] = x$ and $\tau_h(b,a) = (-b,a)$.
A pair $(b,a) \in G$ acts on $\R$ by the affine map
$(b,a).x = b + a x$ and so does the complex affine group 
$\Aff(\C) \cong \C \rtimes \C^\times$ on the complex line~$\C$.
The antiholomorphic involution $\sigma(b,a) = (\oline b, \oline a)$
satisfies $\Aff(\C)^\sigma = \Aff(\R) \cong \R \rtimes \R^\times$,
and $G$ is the identity component of this group. Note that
$G_\C \cong \tilde\Aff(\C) \cong \C \rtimes \C$ with the universal map 
\[ \eta_G \: G \to G_\C, \quad \eta_G(b,a) = (b, \log a).\]
The antiholomorphic extension of $\tau_h$ to $\Aff(\C)$, is given by 
  \[  \oline\tau_h(b,a) = (-\oline b, \oline a)
    \quad \mbox{ with } \quad \Aff(\C)^{\oline\tau_h} = \ie \R \rtimes \R^\times.\]
The subset $\exp(\cS_{\pm \pi/2}h) = \C_r \subeq \C^\times$ is the right half-plane
  in the complex dilation group.

First, we consider in $\Aff(\C)$ the domain
\begin{equation}
  \label{eq:xi1-axb}
 \Xi_1 := \C \times \C_r, \quad \C_r = \{ z \in \C \: \Re z > 0\},  
 \quad \mbox{  with } \quad \Xi_1^{\oline\tau_h} = \ie \R \times \R_+
 = (\Aff(\C)^{\oline\tau_h})_e.
\end{equation}
Then $W_1^c := \Xi_1^{\oline \tau_h}$ satisfies 
$\exp(\cS_{\pm \pi/2} h) W_1^c \subeq  \Xi_1,$ 
so that (Cr2) is satisfied. Further $\eta_G$
extends to a biholomorphic map
\[ \eta_{\Xi_1} \: \Xi_1 \to
\C\times\cS_{\pm\pi/2} \subeq  G_\C \cong \C \rtimes \C, \quad
  \eta_G(b,a) \mapsto (b, \log a). \]
We shall see below that $\Xi_1$ is too large for our purposes 
(Theorem~\ref{thm:hardy}). 
A natural strategy to find better crown domains
is inspired by Riemannian symmetric spaces (cf.~Theorem~\ref{thm:gss} below).

The group $\Aff(\C)_{\oline\tau_h}$
  acts naturally  on $M := \C$, where $\oline\tau_h.z = - \oline z$.
  The two domains $\C_{\pm}$ (upper and lower half plane)
  are invariant under the real group $G_{\tau_h}$, and
  \[ (\C_\pm)^{\oline\tau_h} =  \pm \ie \R_+ \]
  satisfy $\exp(\cS_{\pm \pi/2} h).(\pm \ie \R_+) = \C_\pm$.
  For $m_0 = \pm r \ie$, $r > 0$, we thus obtain with Lemma~\ref{lem:Mxi}
  a crown domain 
  \begin{align*}
 \Xi_{\pm, r} 
    &  := \{ (b,a) \in \Aff(\C) \: b \pm r a\ie \in  \C_{\pm}\}
= \{ (b,a) \in \Aff(\C) \: \pm r^{-1} b + a\ie \in \C_+\}\\
& = \{ (b,a) \in \Aff(\C) \: \pm r^{-1} \Im b + \Re a > 0 \}.
  \end{align*}
Conjugation with $G$ changes the parameter $r$, so that
  it suffices to consider the domains $\Xi_{\pm, 1}$.

To connect with crowns of symmetric spaces,
we consider $\C_+$ as a real $2$-dimensional
  homogeneous space of $G$ via the orbit map
  $(b,a) \mapsto (b,a).\ie = b + a \ie$. 
It has a ``complexification''
  \[ \eta_{\C_+} \:  \C_+ \to \C_+ \times \C_- \subeq \C^2,
    \quad \eta_{\C_+}(z) = (z,\oline z).\]
 The complex Lie group $\Aff(\C)$  acts naturally on
  $\C \times \C$ by the diagonal action with respect to the canonical
  action on $\C$ by affine maps.
  We consider the complex manifold $\Xi_{\C \times \C} := \C_+ \times \C_-$ 
  as a 
  crown domain of the upper half plane
  $\C_+ \cong \eta_{\C_+}(\C_+)$, which is a Riemannian
  symmetric space of $\SL_2(\R)$, 
  acting by M\"obius transformation; 
  see \cite{Kr09}  and \cite[\S 2.1]{Kr08}).
  It carries the antiholomorphic involution
  $\tau(z,w) := (-\oline z, -\oline w)$
  and it is invariant under the real affine group
  $G = \R \rtimes \R_+$, so that we obtain with $\tau$ an action
  by $G_{\tau_h}$ on $\Xi_{\C \times \C}$. 
  As $\C_+ = G.\ie$, the
  corresponding crown domain in $\Aff(\C)$ in the sense of
  Lemma~\ref{lem:Mxi} with $m_0 = \ie$ is 
  \begin{align*}
    \Xi_2 
& := \{ g \in \Aff(\C) \:  g.\eta_{\C_+}(\ie) \in \C_+ \times \C_-\}
    = \{ (b,a) \in \Aff(\C) \: b \pm a\ie \in \C_\pm \} \\
&= \Xi_{+,1} \cap \Xi_{-,1}
  =  \{ (b,a) \in \Aff(\C) \: |\Im b| < \Re a \}.
  \end{align*}
  For this domain
\[     \Xi_2^{\oline\tau_h}  
  = \{ (\ie c,a) \in \ie \R \times \R_+ \: |c| < a \}, \]
and, for $|t| < \pi/2$ and $(\ie c,a) \in \Xi_2^{\oline\tau_h}$, we have
$e^{\ie th}(\ie c,a) = (e^{\ie t}\ie c, e^{\ie t} a)$ with
\[ |\Im(e^{\ie t}\ie c)| = \cos(t) |c|
  < \cos(t) a = \Re(e^{\ie t}a).\]
This proves (Cr2) for $\Xi_2$ and $W^c_2 :=  \Xi_2^{\oline\tau_h}$.
As we shall see in Subsection~\ref{subsec:versus} below, 
for purposes of analytic extension of orbit maps,
the smaller domain $\Xi_2$ behaves much better than~$\Xi_1$,
and the larger domains $\Xi_{\pm, 1}$ also work
for representations satisfying the spectral condition
$\mp \ie \partial U(x) \geq 0$. 
\end{ex}

\begin{ex} {\rm(Complex Olshanski semigroups)}
  \mlabel{ex:olshanski}
  Let $G$ be a connected Lie group for which
  $\eta_G$ is injective and $G_\C$ is simply connected,
  so that we may assume that $G \subeq G_\C$, 
  and let $h \in \g$ be an Euler element.  We assume that 
  $C \subeq \g$ is a pointed closed convex $\Ad(G)$-invariant cone,
  satisfying  $-\tau_h(C) = C$, and
  that the ideal $\fm := C - C \trile \g$ satisfies
  $\g = \fm + \R h$. 
  If $\fm = \g$, we replace $\g$ by $\fm \rtimes_{\ad h} \R$, so that we
  may always assume that $\g = \fm \rtimes \R h$. 
  
  Let $M \trile G$ 
  and $M_\C \trile G_\C$ be the normal integral subgroups corresponding to $\fm$ and $\fm_\C$, respectively. 
  Since $G_\C$ is simply connected, 
  it follows from  \cite[Thm. 11.1.21]{HiNe12} that the subgroup
  $M_\C\subseteq G_\C$ is closed and simply connected,
  and this implies that $M \subeq M_\C$ is closed in $G$,
    as the identity component of the group of fixed points of the
    complex conjugation on $M_\C$ with respect to $M$. 
Then the construction in Step 1 of the proof of \cite[Thm. 15.1.4]{HiNe12} 
shows that the inclusion map $M\into M_\C$ is the universal complexification of $M$. 
Using again that 
$G_\C$ is simply connected, it follows from  
\cite[Prop.~11.1.19]{HiNe12} that   $G_\C \cong M_\C \rtimes_\alpha \C$ with
  $\alpha_z(g) = \exp(zh) g \exp(-zh)$, and thus
  \[ G \cong  M \rtimes_\alpha \R \quad \mbox{ with }  \quad
    \alpha_t(g) = \exp(th) g \exp(-th).\]
We consider the   complex Olshanski semigroup 
  \[ S := M \exp(\ie C^\circ) \subeq M_\C,\]
for which the multiplication map
  $M \times C^\circ  \mapsto S, (g,x) \mapsto g \exp(\ie x)$ is a
  diffeomorphism
(\cite[\S XI.1]{Ne00}). 
Recall from \cite[Lemma 3.2]{NOO21} (and \cite{Ne25}) that
  \[ C \subeq C_+ + \g_0(h)  - C_-
    \quad \mbox{ and } \quad C_+^\circ - C_-^\circ = C^\circ \cap \g^{-\tau_h^\g} \quad \mbox{ for }  \quad C_\pm := \pm C \cap \g_{\pm 1}(h)\]
follow from the invariance of $C$ under $e^{\R \ad h}$ and
$-\tau_h$.
  On the complex manifold $S$, the group $G_{\tau_h}$ acts  by
  \[ (g,t).m = g\alpha_t(m) \quad \mbox{ and } \quad
    \tau_h.m = \oline\tau_h(m)\]
where we use the abuse of notation $\tau_h=(e,-1)\in G_{\tau_h}$ introduced after \eqref{eq:gtauh}. 
 The fixed points of $\oline\tau_h$ form the semigroup 
$S^{\oline\tau_h} = M^{\tau_h} \exp(\ie(C_+^\circ -C_-^\circ)).$ 
  For $x_{\pm 1} \in C_\pm^\circ$, we consider the element 
$s_0 := \exp(\ie(x_1 - x_{-1})) \in S^{\oline\tau_h}$ 
  and in $G_\C$ the crown domain
  \[ \Xi = \{ (g,z) \in M_\C \times \cS_{{\pm \pi/2}} \:
    (g,z).s_0 = g \alpha_z(s_0) \in S \}.\] 
  We now verify (Cr1-2); (Cr3) is redundant because $\Xi \subeq G_\C$.

  \nin (Cr1): First we observe that $\oline\tau_h(s_0) = s_0$,
  and that   $\oline\tau_h(S) = S$
  follows from $-\tau_h^\g(C) = C$. Therefore 
  \[ \oline\tau_h((g,z).s_0) = \oline\tau_h(g,z).s_0\]
implies that the domain $\Xi$ is invariant under
  $\oline\tau_h(g,z) = (\oline\tau_h(g), \oline z)$.

  \nin (Cr2): We recall from \cite[Thms.~2.16, 2.21]{Ne22}
  that, for $y = y_1 - y_{-1} \in C_+^\circ - C_-^\circ$, 
  we have
  \[ \alpha_{\ie t}(\exp(\ie y)) = \exp(\ie(e^{\ie t} y_1 - e^{-\ie t} y_{-1})) \in S
\quad \mbox{ for }  \quad |t| < \pi/2. \]
So we find with 
  \[ W_M^c := M^h_e\exp(\ie(C_+^\circ + C_-^\circ)) s_0^{-1}
    \subeq M_\C^{\oline\tau_h} \]
  that
  \[ W^c := G^h_e (W_M^c \times \{0\}) \subeq \Xi^{\oline\tau_h}
    \quad \mbox{ satisfies }\quad
    \alpha_{\ie t}W^c \subeq \Xi \quad \mbox{ for } \quad |t| < \pi/2.\]
  This proves (Cr2).
\end{ex}

\begin{rem} The construction in Example~\ref{ex:affine-group}
  can also be viewed as a special  case of Example~\ref{ex:olshanski}.
  To see this, write
\[ G = \Aff(\R)_e = \R \rtimes \R_+ \cong M \rtimes \R_+
  \subeq M_\C \rtimes \C^\times
  \quad \mbox{ with } \quad M = \R\times \{1\}.\] 
The invariant  cone $C := \R_+ x$ generates the ideal $\fm = \R x$,
and $C = C_+$.   Then $S = \C_+$ with $S^{\oline\tau_h} = \ie \R_+$.
  For $s_0 = \ie r$, $r > 0$, we obtain
  \[ \Xi = \{ (b,a) \in M \times \C^\times  \: (b,a).s_0
    = b + r a \ie  \in \C_+ \} = \Xi_{+,r}.\] 
\end{rem}

\section{From crown domains to nets of real subspaces}
\mlabel{sec:2}

In this section, we construct nets of
real subspaces satisfying (Iso), (Cov), (RS) and (BW) for all
antiunitary representations $(U,\cH)$ of $G_{\tau_h}$,
  for which the space
  $\cH^\omega(\Xi) \cap \cH^J_{\rm temp}$ is $G$-cyclic (Theorem~\ref{thm:4.9}).

\subsection{$\Xi$-analytic vectors} 

Let $(U,\cH)$ be a unitary representation of $G$
  and $U_h(t) = U(\exp th)$.
  We write $\cH^\omega(\Xi) \subeq \cH^\omega$ 
for the subspace of those analytic vectors $\xi$ 
whose orbit map
\[ U^\xi \: G \to \cH, \quad U^\xi(g) := U(g)\xi \]  
extends from the real
submanifold $G \subeq \Xi$ to a holomorphic map $\Xi \to \cH$.
We then call $\xi$ a {\it $\Xi$-analytic vector.}
For the unitary one-parameter group $U_h(t)= U(\exp th)$,
we consider the spaces
  \[ \cH^\omega_{U_h} \subeq \cH^\infty_{U_h} \subeq \cH
    \subeq \cH^{-\infty}_{U_h} \subeq \cH^{-\omega}_{U_h} \]
  and note that
  \[ \cH^\omega_U \subeq \cH^\omega_{U_h}, \quad
    \cH^\infty_U \subeq \cH^\infty_{U_h}, \quad 
    \cH^{-\infty}_{U_h} \subeq \cH^{-\infty}_U, \quad \mbox{ and } \quad
    \cH^{-\omega}_{U_h}  \subeq  \cH^{-\omega}_U \]
(see Appendix~\ref{app:c} for more details on analytic vectors
and hyperfunction vectors). 

\begin{defn} \mlabel{def:kms}  (KMS condition) Assume 
that $(U,\cH)$ extends to an antiunitary representation of
$G_{\tau_h}$ with $U(\tau_h) = J$.
Then we write 
 $\cH^{-\omega}_{U_h, {\rm KMS}}$  for the real linear
 subspace of $\cH^{-\omega}_{U_h}$, consisting of the hyperfunction 
 vectors $\eta \in \cH^{-\omega}_{U_h}$ 
 for which the orbit map $U_h^\eta(t) = U_h^{-\omega}(t)\eta$ extends to a
\break {weak-$*$} continuous map $U_h^\eta \: \oline{\cS_\pi} \to \cH^{-\omega}_{U_h}$,  
weak-$*$ holomorphic on the interior, 
such that $U_h^\eta(\pi \ie) = J \eta$
(cf.~\cite[\S 2.2, Lemma~2]{FNO25a}, \cite{BN24}). 
We likewise define $\cH^{-\omega}_{\rm KMS}$, $\cH^{-\infty}_{\rm KMS}$, 
and $\cH^{-\infty}_{U_h, {\rm KMS}}$,
where we always consider the corresponding weak-$*$ topology. 
\end{defn}

\begin{defn} \mlabel{def:kmsn}  (Tempered vectors)
  We write $\cH^J_{\rm temp}$ for the real linear subspace
of those $v \in \cH^J$ for which the orbit map $U_h^v(t) := U(\exp th)v$
  extends analytically to $\cS_{\pm \pi/2}$ and
  \[ \beta^{\pm}(v)
    := \lim_{t \to \mp\pi/2} \ee^{\ie t \partial U(h)} v 
    = \lim_{t \to \mp\pi/2} U^v_h(\ie t) \] 
  exist in $\cH^{-\infty}_{U_h}$.
As $Jv= v$, both limits exist if one does.  
In view of \cite[Thm.~6.1]{FNO25b}, this is equivalent to
  the existence of $C, N > 0$ such that
  \begin{equation}
    \label{eq:growthcond}
 \| U^v_h(\ie t)\|^2 \leq C \big(\frac{\pi}{2}-|t|\big)^{-N}
 \quad  \mbox{  for } \quad |t| < \pi/2.
  \end{equation}
By \cite[\S 2.2, Lemma~2]{FNO25a}, the limit $\beta^+(v)$ always
exists in the larger space $\cH^{-\omega}_{U_h, {\rm KMS}}$. 
\end{defn}

\begin{rem} If $h \in \g$ is an Euler element
    and 
   $\g_1(h)\cup\g_{-1}(h)\ne\{0\}$, 
    then \cite[Thm.~7.8]{BN24} implies that, whenever
    there exists an analytic vector of $G$ whose orbit map
    extends to the strip $\exp(\cS_{\pm r} h)$, then
    $r \leq \pi/2$. This motivates the choice of the strip width
    in condition (Cr2) and implicitly in Definition~\ref{def:kmsn}. 
\end{rem}

\begin{rem}
  \label{invar_rem}
  Suppose that $\eta_G \: G \to G_\C$ is injective,
    so that $\Xi \subeq G_\C$. 
    If $v\in\cH$ and $g_0\in G$, then, for every $g\in G$, we have
  \[ U^{U(g_0)v}(g)=U(g)U(g_0)v=U(gg_0)v.\] 
This implies that if $v\in\cH^\omega(\Xi)$ then $U(g_0)v\in\cH^\omega(\Xi g_0^{-1})$. 
Thus, if $\Xi$ is right-invariant with respect to a subgroup $L \subeq G$,
        then the linear subspace $\cH^\omega(\Xi)$ is $L$-invariant.
	
	One may expect that $\Xi$ is right-invariant under 
  a non-trivial complexified subgroup $L_\C \subeq G_\C$,
  as in the semisimple case for $L = K$ (see Subsection~\ref{subsec:semisim}).
Natural candidates for $L \subeq G$ are integral subgroups 
whose Lie algebra $\fl$ is compactly embedded, i.e., 
the subgroup of $\Aut(\g)$, generated by $e^{\ad \fl}$, has compact closure. 
\end{rem}

\begin{lem}   \label{invar_left_lem}
The following assertions hold: 
\begin{enumerate}[{\rm (i)}]
\item\label{invar_left_rem_item1}
	 For arbitrary $y\in\g$, the corresponding Lie derivative operator
	 $$L_y\colon C^\infty(\Xi,\cH)\to C^\infty(\Xi,\cH), \quad 
	 (L_yf)(g):=\frac{d}{dt}\Big|_{t = 0}\ f(\exp(ty)g)$$
	 satisfies $L_y(\cO(\Xi,\cH))\subseteq\cO(\Xi,\cH)$. 
\item\label{invar_left_rem_item2} 
  We have $\dd U(\g)\cH^\omega(\Xi)\subseteq\cH^\omega(\Xi)$ 
and, for $x \in \g_\C,\ p\in\Xi,\ v\in\cH^\omega(\Xi)$, 
\begin{equation}\label{invar_left_rem_eq1} 
U^{\dd U(x)v}(p)=\dd U(\Ad(\eta_\Xi(p))x) U^v(p)
\end{equation}
where $U^v\in\cO(\Xi,\cH)$ is the holomorphic extensioon of the orbit map $U^v\colon G\to\cH$. 
\item\label{invar_left_rem_item3}  
The closure of $\cH^\omega(\Xi)$ in $\cH$ is $U(G)$-invariant. 
If, in particular, $U$ is irreducible and $\cH^\omega(\Xi)$ is
non-zero,  then $\cH^\omega(\Xi)$ is dense in~$\cH$. 
\end{enumerate}
\end{lem}

\begin{prf}
  \ref{invar_left_rem_item1}
  	The operator $L_y$ is the Lie derivative with respect to a 
    fundamental vector field defined by the action of $G$ by left-translations 
on $\Xi$. 
    As this vector field is holomorphic, it preserves on each open subset
    the subspace of holomorphic functions.

    \ref{invar_left_rem_item2} 
Let $v\in\cH^\omega(\Xi)$ and $x\in\g$. 
	For arbitrary $g\in G$ we have 
	$U(g)\partial U(x)v=\partial U(\Ad(g)x)U(g)v$, so that
        \eqref{invar_left_rem_eq1}  follows 
        for $p=g\in G$. 
        The general case $p\in\Xi$ is then obtained by analytic extension.
	Since the mapping $G_\C\to \g_\C$, $g\mapsto \Ad(g)x$, is holomorphic, 
for any basis $y_1,\dots,y_m$ of $\g$, 
there exist $\chi_1,\dots,\chi_m\in\cO(G_\C)$ with
\[ \Ad(g)x=\chi_1(g)y_1+\cdots+\chi_m(g)y_m
\quad \mbox{ for all } \quad g\in G_\C.\]
By plugging this in \eqref{invar_left_rem_eq1}, it
suffices to prove that, for every $y\in\g$,
the function $w\mapsto  \partial U(y)U^v(w)$       
is holomorphic on $\Xi$. 
 
We now check this last fact. 
From the $G$-action on $\Xi$ it follows that, for every $w\in\Xi$,  
there exists $t_w\in\R_+$, such that for all $t\in (-t_w,t_w)$ we have $\exp(ty)w\in \Xi$, and then $U^v(\exp(ty)w)=U(\exp(ty))U^v(w)$. 
Taking the derivative at $t=0$ in this equality, we obtain 
\[ L_y(U^v)(w):=\frac{d}{dt}\Big|_{t = 0} U^v(\exp(ty)w)=\partial U(y)U^v(w)\]  
for arbitrary $w\in\Xi$, where $L_y\colon C^\infty(\Xi,\cH)\to C^\infty(\Xi,\cH)$ is the Lie derivative operator in the direction $y\in\g$. 
	Since  $L_y(\cO(\Xi,\cH))\subseteq \cO(\Xi,\cH)$ 
    by \ref{invar_left_rem_item1},  we are done.
    
    \ref{invar_left_rem_item3} In view of \ref{invar_left_rem_item2}, 
    the closure of $\cH^\omega(\Xi)$ is $U(G)$-invariant by 
	\cite[Cor. to Th. 2, pp.~210--211]{HC53} or
        \cite[Prop.~4.4.5.6]{Wa72}.
	Therefore, if the representation $U$ is irreducible
        and $\cH^\omega(\Xi)\ne\{0\}$, then
        the linear subspace $\cH^\omega(\Xi)$ is dense in~$\cH$. 
\end{prf}

\subsection{Analytic vectors and boundary values} 

As before, $(U,\cH)$ is  an antiunitary representation
of $G_{\tau_h}$ and $J = U(\tau_h)$.
We will establish in Lemma~\ref{lem:beta+-equiv}  a natural 
	$\dd U(\g)$-equivariance property of the mapping 
	$\beta^+\colon\cH^\omega_{U_h}(\cS_{\pm\pi/2})\to\cH^{-\omega}_{U_h}$. 
	As the domain of definition and the range of $\beta^+$ may  not be $\dd U(\g_\C)$-invariant, 
	such an equivariance property does not make sense on $\cH^\omega_{U_h}(\cS_{\pm\pi/2})$. 
	However, the axiom (Cr2) ensures that  $\cH^\omega(\Xi)\subeq\cH^\omega_{U_h}(\cS_{\pm\pi/2})$, 
	and on the other hand $\cH^\omega(\Xi)$ does carry the action of $\dd U(\g_\C)$ by Lemma~\ref{invar_left_lem}\ref{invar_left_rem_item2}, 
	so we study the $\dd U(\g_\C)$-equivariance property for the restriction of $\beta^+$ to  $\cH^\omega(\Xi)$.

\begin{lem} \mlabel{lem:beta+-equiv}
  The map
  \[ \beta^+ \: \cH^\omega(\Xi) \to \cH^{-\omega}, \quad 
    \beta^+(v) := \lim_{t \to \pi/2} U^v(\exp(-\ie th)) 
= \lim_{t \to \pi/2} e^{-\ie t \partial U(h)}v  \]
  satisfies the following equivariance relation with respect to the
  action of $\g_\C$ on both sides:
  \begin{equation}
    \label{eq:beta+rel1}
 \beta^+ \circ \dd U(x) = \dd U^{-\omega}(\zeta(x)) \circ \beta^+
    \quad \mbox{ for } \quad
    \zeta := e^{-\frac{\pi \ie}{2} \ad h} \in \Aut(\g_\C), x \in \g_\C.
  \end{equation}
\end{lem}

\begin{prf}
  (cf.\ \cite[\S 3, Prop.~7(d)]{FNO25a} for the semisimple case)
  Let $v \in \cH^\omega(\Xi)$ and $x \in \g_\C$. Then
  the continuity of the operators $\dd U^{-\omega}(z)$, $z \in \g_\C$,
    implies 
\begin{equation*}
   \dd U^{-\omega}(\zeta(x))\beta^+(v) 
  =  \lim_{t \to \pi/2}  \dd U^{-\omega}(\zeta(x)) U^v(\exp(-\ie t h)).
\end{equation*}
From \eqref{invar_left_rem_eq1} in Lemma~\ref{invar_left_lem}, we then obtain
\begin{equation}
  \label{eq:du-uv}
\dd U(x) U^v(p) = U^{\dd U(\Ad(\eta_\Xi(p))^{-1} x)v}(p)
 \quad \mbox{ for } \quad p \in \Xi.
\end{equation}
This formula holds obviously for $p \in G$, and for
general $p$ it follows by analytic continuation, using that
\[ p \mapsto \dd U(\Ad(\eta_\Xi(p))^{-1} x)v \in \dd U(\g_\C) v \]
is a holomorphic function with values in a finite-dimensional space.
The relation~\eqref{eq:du-uv} implies 
\[ \dd U(\zeta(x)) U^v(\exp(-\ie t h))
  = U^{\dd U(e^{\ie t \ad h} \zeta(x))v}(\exp(-\ie th)) 
= e^{-\ie t \partial U(h)} \dd U\big(e^{\ie t \ad h}\zeta(x)\big)v.\]
We now observe that
$t \mapsto   \dd U\big(e^{\ie t \ad h}\zeta(x)\big)v$ is a continuous curve
in the finite-dimensional subspace $\dd U(\g_\C)v$, so that
we obtain the limit 
\[\lim_{t \to \pi/2}  \dd U(\zeta(x)) U^v(\exp(-\ie t h))
  =\beta^+(\dd U\big(\zeta^{-1}\zeta(x)\big)v)
  =\beta^+(\dd U(x)v).\qedhere\]
\end{prf}

One obtains a natural generalization of the preceding lemma
  by considering the map $\beta^+$ on the subspace
  $\cH_1 := (\cH^\infty)_{U_h}^\omega(\cS_{\pm \pi/2})$ of those smooth vectors
  $\xi \in \cH^\infty$ whose $U_h$-orbit map $U^\xi_h$ extends to a
  holomorphic map $\cS_{\pm \pi/2} \to \cH^\infty$. This space is much
  larger than $\cH^\omega(\Xi)$ and it does not depend on $\Xi$.
  However, the proof   of the lemma shows that this space is $\g$-invariant and
  that the boundary value map
  $\beta^+ \: \cH_1 \to \cH^{-\omega}$ also satisfies 
the covariance relation \eqref{eq:beta+rel1}.

\begin{lem} \mlabel{lem:2.11}
    The following assertions hold:
    \begin{itemize}
    \item[\rm(a)]     The subspace $\cH^\omega(\Xi)$ is $J$-invariant
      and spanned over $\C$ by its $J$-fixed elements. 
\item[\rm(b)] If $v \in \cH^\omega(\Xi)$ is such that
  $\beta^+(v)$ exists in $\cH^{-\infty}_{U_h}$, then also
  $\beta^-(v)$ exists. 
  If $v\in\cH^J$, then the existence of $\beta^+(v)$ implies that
  $v\in\cH^J_{\rm temp}$.
    \end{itemize}
\end{lem}

\begin{prf} (a) Since $U$ is an antiunitary representation
    of $G_{\tau_h}$, we have $J U(g) J = U(\tau_h(g))$, so that we obtain
for $v \in \cH^\omega(\Xi)$ the relation
$J U^v(g) = U^{Jv}(\tau_h(g))$ for $g \in G$, 
hence, by analytic continuation, 
\begin{equation}
  \label{eq:j1}
  J U^v(g) = U^{Jv}(\oline\tau_h(g)) \quad \mbox{ for } \quad g \in \Xi.
\end{equation}
This implies that $Jv \in \cH^\omega(\Xi)$ with 
\begin{equation}
  \label{eq:3}
  U^{Jv}(g) = J U^v(\oline\tau_h(g)) \quad \mbox{ for } \quad g \in \Xi,
  \quad \mbox{ resp.} \quad
  U^{Jv} = J U^v \circ \oline\tau_h.
\end{equation}
Therefore $\cH^\omega(\Xi)$ is $J$-invariant, hence spanned over $\C$
by its $J$-fixed elements. 

\nin (b) Here it suffices to assume that $G = \R$ and $\Xi = \cS_{\pm \pi/2}$.
In this case $\oline\tau_h(z) = \oline z$, so that
\eqref{eq:3} shows that
$\beta^{-}(Jv) = J\beta^{+}(v)$, and this implies the first part of~(b).
If $Jv = v$, then this relation shows that
the existence of $\beta^+(v)$ in $\cH^{-\infty}_{U_h}$ 
implies the existence of $\beta^-(v)$,
hence that  $v\in\cH^J_{\rm temp}$.
\end{prf}

We have $\cH^J_{\rm temp}\subseteq \cH^\omega_{U_h}(\cS_{\pm \pi/2})$,
and the following proposition provides a useful
characterization of $\cH^J_{\rm temp}$ as a subspace,
together with important properties of its elements. 

\begin{prop} \mlabel{prop:temp1} 
The  following assertions hold:
  \begin{itemize}
  \item[\rm(a)] $v \in \cH^J_{\rm temp}$ if and only if
    the orbit map $U_h^v(t) = U_h(t) v$ extends to an $\cH$-valued analytic function on $\cS_{\pm\pi/2}$,  and
    there exist $C,N > 0$ such that 
  $\|\ee^{\pm \ie t \partial U(h)}v\| \leq C \big(\frac{\pi}{2}-|t|\big)^{-N}$
  for $|t| < \pi/2$.   
\item[\rm(b)] $\beta^+ \: \cH^J_{\rm temp} \to \cH^{-\infty}_{U_h, {\rm KMS}}$
  is a real linear bijection (cf.~{\rm Definition~\ref{def:kms}}).
\item[\rm(c)]If $v\in\cH^J$, then $v\in \cH^J_{\rm temp}$ if and only if
 $U_h^v$ extends to a {weak-$*$} continuous function
  $\oline{\cS_{\pm\pi/2}} \to \cH^{-\infty}_{U_h}$,
that is  weak-$*$ holomorphic on the interior. 
\item[\rm(d)] 
  For $v \in \cH^J_{\rm temp}$, and the space $\cH^{-\infty}$
    for the representation $U$ of $G$ instead of $\cH_{U_h}^{-\infty}$, 
the extended orbit map
$U_h^v \colon \oline{\cS_{\pm\pi/2}} \to \cH^{-\infty},
  z \mapsto \ee^{z \partial U(h)} v$
    is continuous and holomorphic on the interior, with respect to the
  weak-$*$ topology on $\cH^{-\infty}$.
  \end{itemize}
\end{prop}

\begin{prf} (a) and (b) are contained in \cite[Thm.~6.1, Prop.~6.2]{FNO25b}.
	
	\nin \rm(c) The assertion follows by~(b). 
	(See also Step~(c) in the proof of \cite[Prop.~6.2]{FNO25b}.)

  \nin (d)  In view of (c), the assertion holds
  with $\cH^{-\infty}_{U_h}$ instead of $\cH^{-\infty}$,
  so that it follows by the continuity of the inclusion 
  $\cH^{-\infty}_{U_h} \into \cH^{-\infty}$.
\end{prf}

\subsection{The Reeh--Schlieder property} 

Let $(U,\cH)$ be an antiunitary representation of
$G_{\tau_h}$ and  suppose that the real linear subspace 
$\sF \subeq \cH^J_{\rm temp} \cap \cH^\omega(\Xi)$ is $G$-cyclic. 
For $v \in \cH^\omega(\Xi) \cap \cH^J_{\rm temp}$, 
relation \eqref{eq:j1} asserts that 
\[ J U^{v}(g) = U^v(\oline\tau_h(g)) \quad \mbox{ for } \quad g \in \Xi. \]
Moreover, the limit
\[ \beta^+(v)
  := \lim_{t \to -\pi/2} \ee^{\ie t\partial U(h)} v 
  = \lim_{t \to -\pi/2} U_h^v(\ie t)  \]
exists in the space $\cH^{-\infty}_{U_h, {\rm KMS}}$ 
by Proposition~\ref{prop:temp1}.  
Here we use that $U^v$ is defined on $\Xi$ and that
$\exp(\cS_{\pm \pi/2} h) \subeq \Xi$ by (Cr2), so that
$U_h^v$ extends to $\cS_{\pm \pi/2}$.

For $\sE := \beta^+(\sF) \subeq \cH^{-\infty}_{\rm KMS}$, we
now consider nets of the form 
\begin{equation}
  \label{eq:HEb}
  \sH_\sE^G(\cO) = \oline{\spann_\R U^{-\infty}(C^\infty_c(\cO,\R))\sE}.
\end{equation}
(see the introduction for more details).
	In this section we show that the subspace $\sH^G_\sE(\cO)$ is total	if $\cO \not=\eset$ (Reeh--Schlieder property).

\begin{thm} {\rm(Reeh--Schlieder Property)} \mlabel{thm:RS}
Let $(U,\cH)$ be an  antiunitary representation of $G_{\tau_h}$ and
\[ \sF \subeq \cH^J_{\rm temp} \cap \cH^\omega \]  
be a real linear subspace such that 
$U(G)\sF$ is total in $\cH$, and
\[ \sE := \beta^+(\sF) \subeq \cH^{-\infty}.\]
Then, for every non-empty open subset
$\cO \subeq G$, 
the real linear subspace $\sH_{\sE}^G(\cO)$ is total in $\cH$,
so that the net $\sH_\sE^G$ satisfies {\rm(Iso), (Cov)} and {\rm (RS)}. 
\end{thm}

\begin{prf} We have to verify hat
    	\begin{equation}
		\label{RSgen_proof_eq1}
  \sH_\sE^G(\cO)^\perp = \{0\} \quad \mbox{ for } \quad \cO \not=\eset.
 \end{equation}
 Let $v \in \cH$. 
We have to show that, if
\begin{equation}\label{eq:skal0}
 0 = \la v, U^{-\infty}(\phi)\beta^+(w) \ra
 =\la U(\phi^*)v, U^w_h\big(-{\textstyle\frac{\pi \ie}{2}}\big) \ra
\end{equation}
holds  for all $\phi \in C^\infty_c(\cO,\R)$ and all $w \in \sF$,
then $v \bot U(G)\sF$. 
Suppose that \eqref{eq:skal0} holds for all
$\phi \in C^\infty_c(\cO,\R)$ and a vector $w\in \sF$. 
For a given $\phi$, there exists an $\eps > 0$ such that 
$\supp(\phi)\exp(th) = \supp(\phi* \delta_{\exp th})
\subeq \cO$ holds for
$|t| < \eps$. 
Then we also have
\[ \la U(\phi^*)v, U^w_h\big(t-{\textstyle\frac{\pi \ie}{2}}\big) \ra
  = \la v, U^{-\infty}(\phi) U(\exp th) U^w_h\big(-{\textstyle\frac{\pi \ie}{2}}\big) \ra
  = \la v, U^{-\infty}(\phi * \delta_{\exp(th)})
  U^w_h\big(-{\textstyle\frac{\pi \ie}{2}}\big) \ra = 0.\]
Uniqueness of analytic continuation now implies that
$\la U(\phi^*)v, U^w_h(z) \ra = 0$ 
holds for all $z \in \cS_{\pm \pi/2}$,
so that we obtain in particular
\[ 0= \la U(\phi^*)v, w \ra  =\la v, U(\phi)w \ra.\]
As $\phi \in C^\infty_c(\cO,\R)$ was arbitrary,
it follows that $v \bot U(\cO)\sF$.
Now $\sF\subeq\cH^\omega$ entails that $v \in (U(G)\sF)^\bot = \{0\}$. 
\end{prf}

\subsection{The Bisognano--Wichmann property}

We continue with the context of the preceding subsection, where
$\sF \subeq \cH^\omega(\Xi) \cap \cH^J_{\rm temp}$ and
$\sE = \beta^+(\sF)$. We now show that
the wedge region  $W^G \subeq G$ from \eqref{eq:defwg} satisfies 
$U^{-\infty}(W^G) \sE \subeq \cH^{-\infty}_{\rm KMS}.$ 
This will imply that 
$\sH^G_\sE(W^G)$ is the standard subspace $\sV$, specified by the operators 
\[   \Delta_\sV = \ee^{2\pi \ie \partial U(h)} \quad \mbox{ and } \quad J_\sV = J
  = U(\tau_h).\]

The following theorem recalls and refines a number of
  powerful tools developed in \cite{BN24} and \cite{FNO25a}. 

\begin{thm} \mlabel{thm:closed-hyperfunc}
  Let $(U,\cH)$ be an antiunitary representation of $G_{\tau_h}$,
  $J := U(\tau_h)$ and $\Delta := \ee^{2\pi \ie \partial U(h)}$.
  Write $\sV := \Fix(J \Delta^{1/2}) \subeq \cH$ for the corresponding
  standard subspace and define KMS vectors
 in $\cH^{-\infty}$ and $\cH^{-\omega}$  with respect to the
 one-parameter group $U_h(t) = U(\exp th)$ as in
   {\rm Definition~\ref{def:kms}}.
  Then the following assertions hold:
\begin{itemize}
\item[\rm(a)] $\cH^{-\omega}_{\rm KMS} =
  \{ \eta \in \cH^{-\omega} \colon (\forall v \in \sV' \cap
  \cH^\omega)\ \eta(v) \in \R\}.$ 
\item[\rm(b)]   $\cH^{-\omega}_{\rm KMS}\cap \cH = \sV.$
\item[\rm(c)] 
  Let $\sV' := J\sV$. Then
  $\cH^{-\infty}_{\rm KMS} = \{ \alpha \in \cH^{-\infty} \:
  \alpha(\sV' \cap \cH^\infty) \subeq \R\}$ and
  $\cH^{-\infty}_{\rm KMS} \cap \cH = \sV.$
\item[\rm(d)]  $\cH^{-\omega}_{U_h, {\rm KMS}} \cap \cH^{-\infty}
  \subeq \cH^{-\infty}_{\rm KMS}$.
\item[\rm(e)]   $\sV \cap \cH^{\infty}$ is a
      $\g$-invariant subspace of $\cH^\infty$, and 
$\cH^{-\infty}_{\rm KMS}$ is a $\g$-invariant subspace of $\cH^{-\infty}$.     
\end{itemize}
\end{thm}

\begin{prf} (a) is \cite[Thm.~2]{FNO25a},  (b) is \cite[Cor.~2]{FNO25a},
  and (c) is \cite[Thm.~6.4, Cor.~6.8]{BN24}. 

\nin (d) \begin{footnote}
    {This proof for general Lie groups
      has been transcribed from \cite[Cor.~3]{FNO25a}, where only 
      semisimple groups and irreducible representations are considered.}
  \end{footnote}
  We recall from (a)   that 
  \[ \cH^{-\omega}_{U_h, {\rm KMS}}
  =  \{ \eta \in \cH^{-\omega}_{U_h} \colon (\forall v \in \sV' \cap
  \cH^\omega_{U_h})\ \eta(v) \in \R\}.\]
  For $\eta \in\cH^{-\omega}_{U_h, {\rm KMS}} \cap \cH^{-\infty}$ we therefore
  have $\eta(v) \in \R$ for $v \in \sV' \cap  \cH^\omega_{U_h}$.
  In view of (c) above, we have to show that
  $\eta(v) \in \R$ holds for any $v \in \sV' \cap \cH^\infty$.

Let $v \in \sV' \cap \cH^\infty$.
  To see that $\eta(v) \in \R$,
  we consider the normalized Gaussians 
$\gamma_n(t) = \sqrt{\frac{n}{\pi}} \ee^{-n t^2}$ 
  and recall from \cite[Prop.~3.5(ii) and \S 5]{BN24} that 
  $v_n := U_h(\gamma_n) v \in \cH^\omega_{U_h}$ 
  converges to $v$ in the Fr\'echet space
  $\cH^\infty$. As $\sV$ is $U_h$-invariant, so is
  $\sV' = J\sV$. Further, 
  $U_h$ commutes with $J$ and $\gamma_n$ is real, so that 
  $v_n \in \sV'$, hence $\eta(v_n) \in \R$ because
  $v_n \in\sV' \cap \cH^\omega_{U_h}$.
  Finally the assertion follows from $\eta(v_n) \to  \eta(v)$.

  \nin (e)  \begin{footnote}
    {The proof of this fact
      in \cite{FNO25a} uses analytic vectors. 
      The proof we give here is more direct.} 
  \end{footnote}
  Let $\eta \in \cH^{-\infty}_{\rm KMS}$,  $x\in \g_{\pm 1}(h)$
   and $\eta' := \dd U^{-\infty}(x)\eta$. 
Then
    \[
U^{\eta'}_h(t) = U^{-\infty}_h(t) \dd U^{-\infty}(x)\eta 
=     \ee^{\pm t} \dd U^{-\infty}(x)U^{-\infty}_h(t) \eta 
=     \ee^{\pm t} \dd U^{-\infty}(x) U^\eta_h(t).\]
As the operator $\dd U^{-\infty}(x) \: \cH^{-\infty} \to \cH^{-\infty}$
is continuous and $U^\eta_h$ extends to a continuous map \break 
$\oline{\cS_\pi} \to \cH^{-\infty}$ 
that is holomorphic on $\cS_\pi$, the same holds for
$U^{\eta'}_h$. Moreover, this extension satisfies
\begin{align*}
 J U_h^{\eta'}(t)
&  = \ee^{\pm t} \dd U^{-\infty}(\tau_h^\g(x)) J U^\eta_h(t)
  = \ee^{\pm t} \dd U^{-\infty}(-x) U^\eta_h(t + \pi \ie)\\
&  = \ee^{\pm (t + \pi \ie)} \dd U^{-\infty}(x) U^\eta_h(t + \pi \ie)
                  = U_h^{\eta'}(t+ \pi \ie).
\end{align*}
This implies the invariance of $\cH^{-\infty}_{\rm KMS}$  under
$\dd U^{-\infty}(\g_{\pm 1}(h))$. 
Since $\cH^{-\infty}_{\rm KMS}$ is obviously invariant under 
$\dd U^{-\infty}(\g_0(h))$, the invariance under $\g$ follows from
$\g = \g_1(h) + \g_0(h) + \g_{-1}(h)$. 
From (c) we obtain
$\sV \cap \cH^\infty
  =  \cH^{-\infty}_{\rm KMS} \cap \cH^\infty,$ 
so this space is also $\g$-invariant.
\end{prf}

Note that Theorem~\ref{thm:closed-hyperfunc}(d) asserts
that, if $\eta\in \cH^{-\omega}_{U_h, {\rm KMS}}$ is a distribution
vector for the $G$-representation $(U,\cH)$, i.e.,
contained in $\cH^{-\infty}$, then the extended orbit map
${U_h^\eta \colon \oline{\cS_\pi} \to \cH^{-\infty}}$ is
automatically weak-$*$ continuous,
and holomorphic on the interior, as an $\cH^{-\infty}$-valued map.

  \begin{rem} For $\phi \in C^\infty_c(\R,\R)$ we have
$U_h^{-\infty}(\phi) \sE \subeq \sV$ 
    because $\sE = \beta^+(\sF) \subeq \cH^{-\infty}_{U_h}$, 
    the closed subspace $\cH^{-\infty}_{\rm KMS}$ is $U_h$-invariant,
    and $\cH^{-\infty}_{\rm KMS} \cap \cH = \sV$
    (Theorem~\ref{thm:closed-hyperfunc}(c)).
    From the $U_h$-equivariance of $\beta^+$ it thus follows for
    $v \in \sF$ that
    \[ U_h(\phi)v \in \cD(\Delta_\sV^{\pm 1/4})
      \quad \mbox{ and }  \quad
      \beta^+(U_h(\phi)v) = \Delta_\sV^{-1/4} U_h(\phi)v
      = U_h(\phi) \beta^+(v) \in \sV.\]    
\end{rem}

\begin{lem} \mlabel{lem:2.16}
Suppose that $W^c = G^h_e \exp(\ie \Omega')$ with an open
convex $0$-neighborhood $\Omega' \subeq \g^{-\tau_h^\g}$
and $W^G = G^h_e \exp(i \zeta(\Omega'))$ (cf.~{\rm Remark~\ref{remW}}).
Then, for $v \in \cH^\omega(\Xi) \cap \cH^J_{\rm temp}$, we have a map 
\[ \Upsilon :=  \beta^+ \circ U^v \: W^c 
  \to \cH^{-\omega}_{U_h, {\rm KMS}} \]
with the following properties:        
\begin{enumerate}[{\rm(a)}]
\item
\label{lem:2.16_a}
It is weak-$*$ analytic.   
\item
\label{lem:2.16_b}
It is $G^h_e$-equivariant. 
\item
  \label{lem:2.16_c}
$\beta^+(U^v(\exp \ie x))
= U^{-\omega}(\exp(x_1 - x_{-1}))\beta^+(v)$
for $x = x_1 - x_{-1} \in \Omega'$.
\item
\label{lem:2.16_d}
$U^{-\omega}(W^G) \beta^+(v) \subeq
                \cH^{-\omega}_{U_h, {\rm KMS}}$.
	\end{enumerate}
\end{lem}

\begin{prf}   
To show that the mapping $\Upsilon$ is defined, we write 
$\alpha_z(p) = \exp(zh)p \in \Xi$ for $p \in W^c, z \in \cS_{\pm \pi/2}$,
which is defined by (Cr2) and Remark~\ref{remW}.
From $\tau_h(p) = p$ and $Jv = v$, it follows that
$J U^v(g) = U^v(\tau_h(g))$ for $g\in G,$ and by analytic continuation
(since both sides are antiholomorphic)
\[ J U^v(m) = U^v(\oline\tau_h(m))\quad \mbox{ for } \quad m \in \Xi.\]
In particular, $U^v(p) \in \cH^J$ and 
\begin{equation}
  \label{eq:j-cov1}
 J U^v(\exp(zh)p) = U^v(\exp(\oline z h)p) \quad \mbox{ for }  \qquad 
 z \in \cS_{\pm \pi/2}.
\end{equation}
Therefore the existence of the limit
$\beta^+(U^v(p)) \in \cH^{-\omega}_{U_h} \subeq \cH^{-\omega}$ 
with respect to the weak-$*$ topology on $\cH^{-\omega}_{U_h}$
(hence also for $\cH^{-\omega}$) follows from  \cite[\S 2.2, Lemma~2]{FNO25a}.
Further,  the relation \eqref{eq:j-cov1} implies that
$ \beta^+(U^v(p)) \in \cH^{-\omega}_{U_h, {\rm KMS}}$ 
(cf.\ Definition~\ref{def:kms}).

We now prove 
\ref{lem:2.16_a}--\ref{lem:2.16_d}.

\nin	\ref{lem:2.16_a}
	Let  
	$m\in W^c$  	and $\xi \in \cH^\omega_{U_h}$.   
	Then there exists $\eps \in (0,\pi/2)$ with 
	$\xi \in \cD(\ee^{-\ie \eps \partial U(h)})$. Let 
	$\xi_\eps := \ee^{-\ie \eps \partial U(h)}\xi.$ 
	We
	then obtain 
	\begin{align*}
		\la \xi, \Upsilon(m) \ra
		&  = \la \ee^{\ie \eps \partial U(h)}  \xi_\eps, \Upsilon(m) \ra 
		= \la \xi_\eps, \ee^{\ie \eps \partial U(h)}  \Upsilon(m) \ra
		= \la \xi_\eps, \Upsilon(\alpha_{\ie\eps} m) \ra\\
		&  = \la \xi_\eps, U^v(\exp(\ie(\eps - \pi/2)h) m) \ra, 
	\end{align*}
	so that the assertion follows from the analyticity of
	$U^v$ on $\Xi$ and the analyticity of the map
	\[  \alpha_{\ie(\eps - \pi/2)} \: W^c \to \Xi.\]
	Finally, we observe that, by Theorem~\ref{thm:closed-hyperfunc}(a),
	the subspace $\cH^{-\omega}_{U_h, {\rm KMS}}$ is closed in 
	$\cH^{-\omega}_{U_h}$, and since $\Upsilon$ takes values in this
	subspace, it is real
	analytic as an $\cH^{-\omega}_{U_h, {\rm KMS}}$-valued map.
	
	\nin 
	\ref{lem:2.16_b}
	follows from the fact that the action of the subgroup
	$(G^h)_e = (G^{\tau_h})_e$ on $\Xi$ preserves the fixed point set 
	$\Xi^{\oline\tau_h}$ and commutes with $U_h(\R)$.
	
	\nin 
	\ref{lem:2.16_c}
From \eqref{eq:beta+rel1} we get inductively that
\begin{equation}
  \label{eq:xn}
 \beta^+ \circ \dd U(\ie x)^n = \dd U^{-\omega}(\ie \zeta(x))^n \circ \beta^+
 \: \cH^\omega(\Xi) \to \cH^{-\omega}.
\end{equation}
Further, $\exp(\ie \Omega') \subeq W^c \subeq \Xi^{\tau_h}$ 
  implies $U^v(\exp(\ie tx)) \in \cH^\omega_{U_h} \cap \cH^J$ for
  $0 \leq t \leq 1$. 
Next we note that the  $\cH^{-\omega}$-valued curves 
  \[ \gamma_1(t) := \beta^+(U^v(\exp t\ie x)) \quad \mbox{ and } \quad
    \gamma_2(t) := U^{-\omega}(\exp(t \ie\zeta(x))) \beta^+(v) \]
  are weak-$*$ analytic.
  For $\gamma_1$ this is a consequence of (a),
  and for $\gamma_2$ it follows from the
  weak-$*$ analyticity of the $G$-orbit maps
  in $\cH^{-\omega}$ (\cite[\S 2.1, Prop.~1(b)]{FNO25a}).
  As both curves have the same Taylor series in $t = 0$,
by \eqref{eq:xn}, 
  uniqueness of analytic extension  implies that 
  \[ \beta^+(U^v(\exp \ie x))
    = U^{-\omega}(\exp(\ie\zeta(x))) \beta^+(v) 
    = U^{-\omega}(\exp(x_1 - x_{-1}))\beta^+(v). \]
\nin 	
	\ref{lem:2.16_d}
        For $p = g \exp(\ie x)$, $g \in G^h_e$, $x = x_1 + x_{-1} \in \Omega'$,
        we obtain with 
        \ref{lem:2.16_c} 
        \begin{align*}
 U^{-\omega}(g \exp(x_1 - x_{-1})) \beta^+(v)
&       = U^{-\omega}(g) \beta^+(U^v(\exp(\ie x)))\\
&        =   \beta^+(U^v(g\exp(\ie x))) \in \beta^+(U^v(W^c)) 
\subeq \cH^{-\omega}_{U_h, {\rm KMS}}.\qedhere
        \end{align*}
\end{prf}

We shall also need the following proposition 
(cf.~\cite[Prop.~9]{FNO25a} for the semisimple case),
which we extend to general Lie groups~$G$, 
to verify that $\sH_\sE(W^G) \subeq \sV$. 

\begin{prop} \mlabel{prop:4.8} For an open subset
  $\cO \subeq G$ and a real subspace $\sE \subeq \cH^{-\infty}$,
  the following are equivalent:
  \begin{itemize}
  \item[\rm(a)]   $\sH_\sE^G(\cO) \subeq \sV$. 
  \item[\rm(b)] For all $\phi \in C^\infty_c(\cO,\R)$ we have
    $U^{-\infty}(\phi)\sE \subeq \sV$. 
  \item[\rm(c)] For all $\phi \in C^\infty_c(\cO,\R)$ we have
    $U^{-\infty}(\phi)\sE \subeq \cH^{-\infty}_{\rm KMS}.$ 
  \item[\rm(d)]  $U^{-\infty}(g) \sE \subeq \cH^{-\infty}_{\rm KMS}$
    for every $g \in \cO$.
  \end{itemize}
\end{prop}

\begin{prf}  It is clear that (a) implies (b) by the   definition
  of $\sH^G_\sE(\cO)$. Further, (b) implies (c)  because
  $\sV \subeq \cH^{-\infty}_{\rm KMS}$
  (Theorem~\ref{thm:closed-hyperfunc}(c)). 
  
  For the implication (c) $\Rightarrow $ (d), let $(\delta_n)_{n \in \N}$ be a $\delta$-sequence in $C^\infty_c(G,\R)$. 
  For $\xi \in \cH^\infty$,
    we then have 
  $U(\delta_n)\xi \to \xi$ in $\cH^\infty$ and hence also in $\cH^{-\infty}$.
  It follows in particular that
  \[U^{-\infty}(\delta_n * \delta_g) \eta
  =  U^{-\infty}(\delta_n) U^{-\infty}(g) \eta \to  U^{-\infty}(g) \eta
  \quad \mbox{ for } \quad \eta\in \cH^{-\infty}.\] 
  Hence the closedness of
  $\cH^{-\infty}_{\rm KMS}$   (Theorem~\ref{thm:closed-hyperfunc}(c)),
  shows that (c) implies (d).
  Here we use that $\delta_n * \delta_g \in C^\infty_c(\cO,\R)$
  for $g \in \cO$ if $n$ is sufficiently large. 

As the $G$-orbit maps in $\cH^{-\infty}$ are continuous
and $\cH^{-\infty}_{\rm KMS}$ is closed, hence stable under integrals over
compact subsets,
and $U^{-\infty}(C_c^\infty (\cO,\R))\cH^{-\infty} \subset \cH^\infty$, we see that (d) implies (a).
\end{prf}

The following theorem is a core result of this paper. 
It provides a general tool to construct
well-behaved nets of real subspaces on Lie groups via crown domains.

\begin{thm} \mlabel{thm:4.9}
  {\rm(Construction Theorem for nets of real subspaces)} 
  Let $(U,\cH)$ be an antiunitary
  representation of $G_{\tau_h} = G \rtimes \{\bone,\tau_h\}$ and
 \[ \sF \subeq \cH^J_{\rm temp} \cap \cH^\omega(\Xi) \] 
 be a $G$-cyclic subspace of $\cH$, i.e., $U(G)\sF$ is total in $\cH$.
 We consider the linear subspace
 \[ \sE = \beta^+(\sF)
   \subeq \cH^{-\infty}.\]
 Then the net $\sH^G_{\sE}$ on $G$ 
 satisfies {\rm (Iso), (Cov), (RS)} and
 {\rm (BW)}, in the sense that
$\sH^G_\sE(W^G) = \sV$  holds for $W^G$ as in \eqref{eq:defwg}. 
\end{thm}

If the unitary representation $(U\vert_G,\cH)$ is
irreducible, then every non-zero subspace \break $\sF \subeq \cH^J_{\rm temp} \cap \cH^\omega(\Xi)$ is $G$-cyclic,  so that
$\sF$ as in the hypothesis exists as soon as
$\cH^J_{\rm temp} \cap \cH^\omega(\Xi)\ne\{0\}$. 

\begin{prf}[Proof of Theorem~\ref{thm:4.9}] 
	(Iso) and (Cov) are trivially satisfied and
  (RS) follows from Theorem~\ref{thm:RS}. It remains to
  verify (BW). 
By Lemma~\ref{lem:2.16}(d), we have
$U^{-\omega}(W^G) \sE \subeq \cH^{-\omega}_{U_h, {\rm KMS}}$,
and $\sF \subeq \cH^J_{\rm temp}$ implies that
$\sE \subeq \cH^{-\infty}_{U_h} \subeq \cH^{-\infty}$.
Now Theorem~\ref{thm:closed-hyperfunc}(d) yields 
\[ U^{-\infty}(W^G) \sE = 
    U^{-\omega}(W^G) \sE
    \ {\buildrel \ref{lem:2.16}(d) \over  \subeq}\
    \cH^{-\omega}_{U_h, {\rm KMS}} \cap \cH^{-\infty}
    \ {\buildrel \ref{thm:closed-hyperfunc}{\rm(d)}
      \over\subeq}\  \cH^{-\infty}_{\rm KMS},\]  
  and therefore Proposition~\ref{prop:4.8}(a),(d) imply
  that   $\sH_{\sE}^G(W^G)  \subeq \sV$. 
  In particular, $\sH_{\sE}^G(W^G)$ is separating. 

  Since $\exp(\R h) W^G = W^G$, the real linear subspace
  $\sH_{\sE}^G(W^G)$ is invariant under the modular group 
  $U(\exp \R h) = \Delta_\sV^{\ie\R}$ of the standard subspace~$\sV$.
  It is cyclic by Theorem~\ref{thm:RS},
  so that \cite[Prop.~3.10]{Lo08} implies that
$\sH_{\sE}^G(W^G)  = \sV$. 
\end{prf}

We note the following immediate corollary: 
\begin{cor} \mlabel{cor:pullback}
  {\rm(The pullback construction)} Let 
 $\phi \:   (G_1, h_1, \Xi_1) \to  (G_2, h_2, \Xi_2)$ 
  be a morphism of   crowned Lie groups
  and $(U,\cH)$ an antiunitary representation of
  $G_{2,\tau_{h_2}}$. Let 
 \[ \sF \subeq \cH^J_{\rm temp} \cap \cH^\omega(\Xi_2) \] 
 be a $G_2$-cyclic subspace of $\cH$
 and consider the antiunitary representation
 $U_1 := U \circ \phi$ of $G_{1,\tau_{h_1}}$ on
 the $G_1$-cyclic subspace $\cH_1 \subeq \cH$
 generated by $\sF$. Then $\sE = \beta^+(\sF) \subeq \cH_1^{-\infty}$
 and the net $\sH^{G_1}_{\sE}$ on $G_1$ satisfies {\rm(Iso), (Cov), (RS)} and
 {\rm (BW)}.  
\end{cor}

Since $\SL_n(\R)$ has a natural crown domain 
$\Xi_{\SL_n(\C)}\subeq\SL_n(\C)$,
  morphisms of crowned Lie groups
  $(G, h, \Xi) \to (\SL_n(\R), h', \Xi_{\SL_n(\C)})$
  can be used to obtain situations where the Construction
  Theorem applies. In Section~\ref{sec:5} we use this technique
in particular for the split oscillator group.

\subsection{Complex Olshanski semigroups}

We return to Example~\ref{ex:olshanski},
  where 
  \[ S = M \exp(\ie C^\circ) \subeq M_\C \]
  is a complex Olshanski semigroup,
  $s_0 := \exp(\ie(x_1 - x_{-1})) \in S^{\oline\tau_h}$
  with $x_{\pm 1} \in C_\pm^\circ$, 
 and assume that $(U,\cH)$ is an antiunitary representation
 of $G_{\tau_h}$ for which
\[ C \subeq C_U \cap \fm, \quad \mbox{ where } \quad
   C_U = \{ x \in \g \: -\ie \cdot \partial U(x) \geq 0\}\] 
is the {\it positive cone of $U$}. 
 This implies that $U$ extends to a strongly continuous contraction representation
 of $\oline S := M \exp(\ie C)$ by
 \[ U(m \exp(\ie x)) = U(m) \ee^{\ie \partial U(x)},\quad  m \in M, x \in C,\]
 and $U\res_{S} \: S \to B(\cH)$ is holomorphic 
with respect to the operator norm topology 
 (cf.\ \cite[Thm.~XI.2.5]{Ne00}). 
 For any $v \in \cH^J$, we thus obtain a holomorphic orbit map
 \[ U^v \: S \to \cH, \quad U^v(s) = U(s)v.\]
 It follows in particular, that the map
 \begin{equation}
   \label{eq:holo1}
   \Xi \to B(\cH), \quad (g,z) \mapsto U^v((g,z).s_0)
 \end{equation}
 is holomorphic.  We claim that 
 \begin{equation}
   \label{eq:temp-xi-1}
U(s_0)\cH^J_{\rm temp} \subeq \cH^\omega(\Xi) \cap \cH^J_{\rm temp} 
\end{equation}
in particular $\cH^\omega(\Xi)\cap \cH_{\rm temp}^J\neq \{0\}$.
 So let $w \in \cH^J_{\rm temp}$ and
 $v := U(s_0)w$. Then
 \[  U^v(g,t)
   = U(g) U_h(t) U(s_0) w 
   = U(g \alpha_t(s_0)) U_h(t) w 
   = U((g,t).s_0) U_h^w(t).\]
 As the evaluation map $B(\cH) \times \cH \to \cH$ is holomorphic,
 the representation $U \: S \to B(\cH)$, and
 thus the orbit map $U^w \: S \to \cH$, are holomorphic, 
for $w \in \cH^J_{\rm temp} \subeq \cH^\omega_{U_h}(\cS_{\pm \pi/2})$, 
 the prescription $U^v(g,z)   := U((g,z).s_0) U_h^w(z)$ 
 defines a holomorphic map $\Xi \to \cH$, extending
 the orbit map $U^v$ on~$G$.
 This proves~$v \in \cH^\omega(\Xi)$.
 To see that also $v \in \cH^J_{\rm temp}$, note that 
\[ U^v(e,t)  = U(\exp th) U(s_0)w = U(\alpha_t(s_0)) U_h^w(t) 
  \quad \mbox{ for }  \quad t \in \R \]
implies by analytic continuation 
\[ e^{\ie t \partial U(h)} v =  U^v(e,\ie t) = U(\alpha_{\ie t}(s_0)) U_h^w(\ie t)
  \quad \mbox{ for } \quad |t| < \pi/2.\]
As $U(S)= U(G) e^{\ie \partial U(C)}$ 
consists of contractions, 
we then obtain
$\|e^{\ie t \partial U(h)} v \| \leq \|U_h^w(\ie t)\|.$ 
It thus follows with \eqref{eq:growthcond} 
that $w \in \cH^J_{\rm temp}$ implies that $v \in \cH^J_{\rm temp}$.  

We conclude that the dense subspace $U(s_0) \cH^J_{\rm temp}$
is contained $\cH^\omega(\Xi) \cap \cH^J_{\rm temp}$, so that this
subspace is dense as well, and thus Theorem~\ref{thm:4.9} applies
to all antiunitary representations $(U,\cH)$ of
$G_{\tau_h}$ with $C \subeq C_U$.

\subsection{Semisimple Lie groups}
\mlabel{subsec:semisim}

The following theorem builds on the analytic extension
results from \cite{KSt04}, the connection with standard
subspaces from \cite{FNO25a} and the extension
of \cite{KSt04} to non-linear groups from~\cite{Si24},
which also contains the result on the temperedness, resp.,
  the growth condition~\eqref{eq:growthcond}. 

\begin{thm} \mlabel{thm:gss} Suppose that $\g$ is semisimple,
  $\g = \fk \oplus \fp$ is a Cartan decomposition, and that 
  $h \in \fp$ an Euler element. We consider two situations: 
  \begin{itemize}
  \item[\rm(a)] If $G \subeq G_\C$,
    $K = \exp_G \fk$ and $K_\C = K \exp(\ie \fk)\subeq G_\C$, 
    then we consider the domain 
\[ \Xi_{G_\C} = G \exp(\ie\Omega_\fp) K_\C \subeq G_\C,
    \quad \mbox{ where } \quad
    \Omega_\fp = \{ x \in \fp \: \Spec(\ad x) \subeq
    (-\pi/2,\pi/2)\},\]
  and
  \begin{equation}
    \label{eq:wc-ss}
    W^c = G^h_e \exp(\ie \Omega_\fp^{-\tau_h^\g}) K_\C^{-\oline\tau_h},
  \end{equation}
\item[\rm(b)] If $G$ is simply connected,
  $K = \exp_{G_\C}(\fk)$, and $K_\C = K \exp_{G_\C}(\ie \fk)$ is the
    universal complexification of $K$,   then we consider 
the  simply connected covering manifold $\Xi := \tilde\Xi_{G_\C}$
and
\begin{equation}
    \label{eq:wc-ssb}
    W^c = G^h_e.\exp(\ie \Omega_\fp^{-\tau_h^\g}). \tilde K_\C^{-\oline\tau_h},
  \end{equation}
  with respect to the $G$-action from the left and
  the $\tilde K_\C$-action from the right. 
  \end{itemize}
  Then, in both cases, {\rm(Cr1-3)} are satisfied,
  and,   for every irreducible antiunitary representation 
  $(U,\cH)$ of $G_{\tau_h}$, 
  the space $\cH^\omega(\Xi) \cap \cH^J_{\rm temp}$
  is dense in~$\cH^J$.
\end{thm}

\begin{prf} (a) As $\Xi_{G_\C}$ is the inverse image of the open crown domain
  \[ \Xi_{G_\C/K_\C} = G \exp(\ie\Omega_\fp)K_\C
    \cong G \times_K \ie \Omega_\fp \]
  in $G_\C/K_\C$,   it is an
  open subset of $G_\C$ which is a holomorphic $K_\C$-principal
  bundle over the contractible space $\Xi_{G_\C/K_\C}$.
  Condition (Cr1) follows from the invariance of
  $G$, $\exp(\ie \Omega_\fp)$ and $K_\C$ under the antiholomorphic
  involution $\oline\tau_h$.
  To verify (Cr2), we first observe that $W^c
  \subeq \Xi_{G_\C}^{\oline \tau_h}$ follows from the fact that all
  $3$ factors in \eqref{eq:wc-ss} consist of fixed points of $\oline\tau_h$.
  The set of $\oline\tau_h$-fixed points in $\Xi_{G_\C}$
  is a fiber bundle over
  \[ \Xi_{G_\C/K_\C}^{\oline\tau_h}
    \cong G^h_e \times_{K^h_e} \ie \Omega_\fp^{-\tau_h^\g}
\into  G_\C^{\oline\tau_h}/K_\C^{\oline\tau_h}\]
  (\cite[Thm.~6.1]{NO23}),
  so that the $K_\C$-principal bundle structure of $\Xi_{G_\C}$
  implies that it is a $K_\C^{\oline\tau_h}$-principal bundle.
  We conclude that $W^c = \Xi_{G_\C}^{\oline \tau_h}$.
The inclusion $\exp(\cS_{\pm \pi/2}h) W^c \subeq \Xi_{G_\C}$ is equivalent to 
$\exp(\cS_{\pm \pi/2}h) \Xi_{G_\C/K_\C}^{\oline\tau_h} \subeq \Xi_{G_\C/K_\C}$,
which has been shown in \cite[\S 8]{MNO24}.

\nin (b) We now assume that $G$ is simply connected, 
  so that $\eta_G \: G \to G_\C$ has discrete kernel and
  $G_\C$ is simply connected. 
  The discussion in the proof of \cite[\S 3, Prop.~5]{FNO25a} shows
  that the simply connected covering
  $\Xi := \tilde\Xi_{G_\C}$ is a complex manifold which is a
  $\tilde K_\C$-principal bundle over the contractible space
  $\Xi_{G_\C/K_\C}$. 
  This implies that 
  $\eta_G \: G \into \Xi_{G_\C}$ lifts to an embedding 
  $G \into \Xi$ and
  $\pi_1(\eta_G(G)) \cong \ker(\eta_G) \cong \pi_1(\Xi_{G_\C})$
  acts as a group  of deck transformations on $\Xi$.
  We thus obtain a free action of $G$ on $\Xi$ on the left, 
  and a free holomorphic action of $\tilde K_\C$ from the right,
  where $\tilde K_\C$ denotes the simply connected
 universal complexification of the integral subgroup
 $\tilde K = \exp_G \fk \subeq G$. 
The exponential map $\ie \Omega_\fp \to \Xi_{G_\C}$ also lifts
  to a map $\exp \: \ie\Omega_\fp \to \Xi_{G_\C}$. 
In this sense, we obtain a factorization 
  \[ \Xi = G \exp(\ie \Omega_\fp) \tilde K_\C. \]
The involution $\tau_h$ of $G$ extends to an antiholomorphic
involution $\oline\tau_h$ on $\Xi$ (by the Lifting Theorem for Coverings),
and we obtain a connected open subset 
  \[ W^c := G^h_e \exp(\ie \Omega_\fp^{-\tau_h^\g}) \tilde K_\C^{\oline\tau_h}
    \subeq \Xi^{\oline\tau_h}. \]
Here $\tilde K_\C^{\oline\tau_h}$ is connected because
  $\tilde K_\C$ is simply connected (\cite[Thm.~IV.3.4]{Lo69}).  

  We thus obtain crowned Lie groups in both situations.
  The Kr\"otz--Stanton Extension Theorem and Simon's generalization \cite{Si24}
  for non-linear groups (needed in case (b)) imply that,
  for every irreducible antiunitary representation
  $(U,\cH)$ of $G_{\tau_h}$, 
  the space $\cH^{[K]}$ of $K$-finite vectors is contained in
  $\cH^\omega(\Xi)$. By \cite[Thm. 3.2.6]{Si24},
   the space $\cH^{[K]} \cap \cH^J$
  of $J$-fixed $K$-finite vectors, which is dense in $\cH^J$, is contained
  in $\cH^J_{\rm temp}$. 
\end{prf}

\subsection{The Poincar\'e group}

We consider the Poincar\'e group
\[ G := \R^{1,d} \rtimes \SO_{1,d}(\R)_e
  \subeq G_\C := \C^{1+d} \rtimes \SO_{1+d}(\C)\]
and the Euler element $h \in \so_{1,d}(\R) \subeq \g$, generating
a Lorentz boost:
\[ h(x_0, x_1, \ldots, x_d) := (x_1, x_0,0,\cdots, 0).\]
The corresponding involution
\[  e^{\pi \ie h} = \diag(-1,-1,1,\ldots, 1) \in \SO_{1,d}(\R) \]
acts by conjugation on $G_\C$ (denoted $\tau_h)$ and 
we also obtain an antiholomorphic involution on $G_\C$ 
by $\oline\tau(g) := \tau_h(\oline g)$. 

We consider the action of $G_{\C,\oline\tau_h}$ on $M = \C^{1+d}$ by affine maps, and
\[ \tau_M(z_0, \cdots, z_d) = (-\oline z_0, -\oline z_1,
  z_2, \cdots, z_d).\]
Write 
\[  V_+ := \{ x = (x_0, \bx) \in \R^{1,d} \: x_0 > \bx^2 \} \]
for the future light cone in Minkowski space $V := \R^{1,d}$.
Then 
\[ \Xi_M := \R^{1,d} + \ie V_+
= \{z \in \C^{1+d} \: \Im z \in V_+ \} \]
is an open tube domain in $M$, invariant under $G_{\tau_h}$,
and the subset of $\tau_M$-fixed points is 
\[ \Xi_M^{\tau_M} = \{ (\ie y_0, \ie y_1, y_2, \ldots, y_d) \:
  y_j \in \R, y_0 > |y_1| \} \subeq i \R^2 \oplus \R^{d-1}.\]
For $z := (\ie y_0, \ie y_1, y_2, \ldots, y_d) \in \Xi_M^{\tau_M}$, we have 
$\Im(e^{\ie th}z) = (\cos(t)y_0, \cos(t)y_1, 0,\ldots,0),$ 
so that
\[ e^{\ie t h} \Xi_M^{\tau_M}  \subeq \Xi_M \quad \mbox{ for } \quad
  |t| < \pi/2.\]
We thus put $W^{M,c} :=  \Xi_M^{\tau_M}$. 

For $m_0 := \ie \be_0 \in \Xi_M^{\tau_M}$ we now obtain a crown domain
\[ \Xi := \{ g \in G_\C \: g.\ie \in \Xi_M \},\]
and we put 
\[ W^c := \{ g \in G_\C^{\oline\tau_h} \: g.\ie \in \Xi_M \} 
  = \{ g \in G_\C^{\oline\tau_h} \: g.\ie \in W^{M,c} \}
  = \Xi^{\oline\tau_h}.\]
By Lemma~\ref{lem:Mxi},
  this data defines a crowned Lie group $(G,h, \Xi)$. 

The unitary representations
of $G$ that are most relevant in Physics can be realized 
on a Hilbert space $\cH$ of holomorphic functions $f \: \Xi_M \to \cK$,
where $\cK$ is a finite-dimensional Hilbert space, and $\cH$
is specified by a reproducing kernel of the form
\begin{equation}
  \label{eq:K(z,w)}
 K(z,w) = \tilde\mu(z - \oline w) =
 \int_{V_+^\star} e^{\alpha(z - \oline w)}\, d\mu(\alpha),
\end{equation}
where $\mu$ is a tempered $\Herm(\cK)_+$-valued measure on the dual cone 
\[ V_+^\star = \{ \lambda \in  V^* \: \lambda(V_+)
\subeq [0,\infty[\}.\]  
Its Fourier transform 
is considered as a holomorphic function $\tilde\mu \: \Xi_M \to B(\cK)$
whose boundary values define an element in $\cS'(\R^{1+d}, B(\cK))$. 
More concretely, there exists a
representation ${\rho \: G \to \GL(\cK)}$ such that
\[ (U(x,g)f)(z)  = \rho(g) f((x,g)^{-1}.z)  = \rho(g) f(g^{-1}.(z-x)),
  \quad  (x,g) \in G, z \in \Xi_M.\] 
We refer to \cite{NOO21} for a detailed discussion of the analytic aspects
of such Hilbert spaces and the standard subspaces associated to~$h$.
We extend $U$ to an antiunitary representation of $G_{\tau_h}$ by
\[ (Jf)(z) := J_\cK f(\tau_M(z)),\]
where $J_\cK$ is a conjugation on $\cK$.
Then $\cH^J$ consists of those functions
with $f(\Xi_M^{\tau_M}) \subeq \cK^{J_\cK}$
(\cite[Lemma~2.5]{NOO21}).

Let
\[ K_z \: \cH \to \cK, \quad K_z(f) := f(z) \]
denote the evaluation operator in $z \in \Xi_M$ and
$K_z^* \: \cK \to \cH$ its adjoint. Then the functions 
\begin{equation}
  \label{eq:kwxi}
  K_w^* \xi, \quad w \in \Xi_M^{\tau_M}, \xi \in \cK^{J_\cK}
\end{equation}
are contained in $\cH^J$ and span a dense subspace thereof
(\cite[Lemma~3.11]{NOO21}).
A straightforward calculation shows that
\[ U(g) K_w^*\xi = K_{g.w}^* \rho(g^{-1})^* \xi.\]
As the representation $\rho$ extends to a holomorphic representation
of $G_\C$, it follows that $K_w^*\xi \in \cH^\omega(\Xi)$ if 
$w \in W^{M,c} = \Xi_M^{\tau_M}$, and thus
all functions \eqref{eq:kwxi} are contained in $\cH^\omega(\Xi)^J$.
To see that they are actually contained in $\cH^J_{\rm temp}$,
we need to estimate the norms
$\|e^{\ie t \partial U(h)} K_w^*\xi\|$ for $|t| \to \pi/2$
(cf.\ Definition~\ref{def:kmsn}).
As $\rho(\exp(\ie th))$ is bounded for $|t| \leq \pi/2$, we have to verify
that
\[ \| K(e^{\ie th}w, e^{\ie th}w)\| \leq C \Big(\frac{\pi}{2} - |t|\Big)^{-N}
\quad \mbox{ for some } \quad C, N > 0.\] The operator
$K(e^{\ie th}w, e^{\ie th}w)$ is the Fourier transform of $\mu$, evaluated
in
\[ e^{\ie th} w - e^{-\ie th} \oline w, \quad
  w = (\ie  y_0, \ie  y_1, y_2, \ldots, y_d),\]
so that
\begin{equation}
  \label{eq:expws}
  e^{\ie th} w - e^{-\ie th} \oline w
  = \cos(th) (\ie y_0 \be_0 + \ie y_1 \be_1) \in \R \ie \be_0 + \R \ie \be_1.
\end{equation}
In view of \cite[Prop.~4.11, \S 2.3]{FNO25a}, the
temperedness of the measure $\mu$ yields an estimate
\[ \|\tilde\mu(x + \ie y)\| \leq C \|y\|^{-N}\quad \mbox{ for } \quad
  x + \ie y \in \R^{1,d} + \ie V_+, \]
and we conclude from \eqref{eq:expws}
that the functions $K_w^*\xi$ are contained in $\cH^J_{\rm temp}$. 
Therefore all our assumptions are satisfied for the finite-dimensional space
\[ \sF := \{ K_{\ie \be_0}^* \xi \:  \xi \in \cK^{J_\cK}\}.\]
We refer to \cite{NOO21} for detailed descriptions of the corresponding
standard subspaces $\sV \subeq\cH$.

\section{The affine group of the real line}
\mlabel{sec:3}

In this section we take a closer look at the affine
group $G \cong \R \rtimes \R_+$ 
of the real line. We shall use the notation
introduced in Example~\ref{ex:affine-group}. 
A careful analysis of its representation $(U,\cH)$ 
on the Hardy space
\[ H^2(\C_+) := \Big\{f \in \cO(\C_+) \:
  \sup_{y > 0} \int_\R |f(x + iy)|^2\, d x < \infty \Big\}\]
on the upper half plane shows that the domain
$\Xi_1 = \C \times \C_r$ is too large for our purposes: 
$\cH^\omega(\Xi_1)$ is dense but it intersects the subspace
$\cH^J_{\rm temp}$ trivially (Theorem~\ref{thm:hardy}).
The smaller convex domain
\[     \Xi_2  = \{ (b,a) \in \Aff(\C) \: |\Im b| < \Re a\}
\subsetneqq\Xi_1 \]
behaves much better than $\Xi_1$ because
  the irreducible unitary representations of $G$ extend
  to unitary representations of $\PSL_2(\R)$, for which $\C_+$ is a
  Riemannian symmetric space. Therefore
Theorem~\ref{thm:gss} implies corresponding results
for~$G$.

\subsection{The Hardy space representation of the affine group}
\mlabel{sec:hardy-affine-grp}

For $G = \Aff(\R)_e$, we consider the irreducible antiunitary representation 
of $G _{\tau_h}$ on the Hardy space $H^2(\C_+)$ and on $L^2(\R_+)$.
On $L^2(\R_+)$ it takes the form
\[ (\widehat{U}_+(b,a)f)(p)  = \ee^{\ie bp}  a^{1/2} f(ap)
  \quad \mbox{ with } \quad
\widehat{U}_+(\tau_h)= \hat J, \quad  (\widehat{J}f)(p) = \oline{f(p)}.\]
On $H^2(\C_+)$, it is given by
\begin{equation}
  \label{eq:Uact}
 (U_+(b,a)F)(z) = a^{-1/2} F\Big(\frac{z+b}{a}\Big)
 \quad \mbox{ with } \quad
U_+(\tau_h) = J, \quad  (JF)(z) = \oline{F(-\oline z)}.
\end{equation}
For the derived representations, this leads to
\[  \dd U_+(1,0) = \frac{d}{dz}  \quad \mbox{ and } \quad 
\dd U_+(0,1) = -\frac{1}{2} \bone - z \frac{d}{dz}.\] 

These representations are intertwined by the Fourier transform
\[ \hat f(z) := \cF(f)(z) := \int_0^\infty \ee^{\ie zp} f(p)\,  dp.\]
The functions $e_{iz}(p) = \ee^{izp}$ are contained in $L^2(\R_+)$
for $z \in \C_+$ and $\cF(f)(z)  = \la e_{-\ie \oline z}, f \ra,$ 
so that the corresponding reproducing kernel of $H^2(\C_+)$ takes the form
\[ K(z,w) = \la e_{-\ie \oline z}, e_{-\ie  \oline w} \ra = \int_0^\infty \ee^{\ie (z-\oline w)p} \,  dp
  = \frac{\ie }{z-\oline w}\]
(cf.~\eqref{eq:K(z,w)}).

In $L^2(\R_+)$, the entire vectors $f$ for the translation
group are specified by the condition $e_z f \in L^2$ for all $z \in \C$,
which implies in particular that the Fourier transform extends to
an entire function on $\C$. As positive translations act on $H^2(\C_+)$
by upward translations, we thus obtain elements
$F \in H^2(\C_+)$ extending to all of $\C$ in such a way that
\[ \|F_y \|_2^2 := \int_{\R} |F(x + \ie  y)|^2\, dx < \infty
  \quad \mbox{ holds for all } \quad y \in \R.\]
For $F= \cF(\phi)$, we concretely have
$\|F_y\|^2 = \int_{\R_+} \ee^{-2yp} |\phi(p)|^2\, dp.$ 

With the reproducing kernel
\begin{equation}
  \label{eq:kwz}
 K(w,z) = \frac{\ie }{w - \oline z} \quad \mbox{ satisfying } \quad
 \|K_z\|^2 = K(z,z) = \frac{1}{2 \Im z},
\end{equation}
we obtain for $z \in \C_+$ and $F\in H^2(\C_+)$ that
\begin{equation}
  \label{eq:fzesti}
  |F(z)| \leq \|F\| \cdot \|K_z\| = \|F\| \cdot \frac{1}{\sqrt{2\Im z}}.
\end{equation}

The functions $(K_z)_{z \in \C_+}$ are analytic vectors for $G$,
depending antiholomorphically on $z$, so that we obtain an
antiholomorphic map $\C_+ \to \cH^\omega, z \mapsto K_z$.
This leads for any hyperfunction vector $\eta \in \cH^{-\omega}$ to a
holomorphic function on $\C_+$ by $\Phi(\eta)(z) := \alpha(K_z).$ 
This realization transforms the $G$-action on $\cH^{-\omega}$
into the action on $\cO(\C_+)$ by
\[ ((b,a).F)(z) = a^{-1/2} F\Big(\frac{z+b}{a}\Big)\]
(cf.\ \eqref{Uact} below). 
In fact,
\begin{equation}
  \label{eq:actonkz}
 (U_+(b,a)K_z)(w)
  = a^{-1/2} K_z\Big(\frac{w+b}{a}\Big)
  = a^{-1/2} \frac{\ie }{\frac{w+b}{a}-\oline z}
  = a^{1/2} \frac{\ie }{w+b-a\oline z}
  = a^{1/2} K_{az - b}(w)
\end{equation}
leads to
\begin{align*}
 \Phi(U_+^{-\omega}(b,a)\eta)(z)
&  = \eta(U_+^\omega(b,a)^{-1}K_z)
    = \eta(U_+^\omega(-a^{-1}b,a^{-1})K_z)\\
&  = a^{-1/2} \eta(K_{a^{-1} z + a^{-1} b})
  = a^{-1/2} \Phi(\eta)((z+b)/a).
\end{align*}
As this is compatible with the action \eqref{eq:Uact}
on $H^2(\C_+)$, we have equivariant inclusions
\[ H^2(\C_+) \subeq H^2(\C_+)^{-\omega} \subeq \cO(\C_+).\]
For $\exp(th) = (0,\ee^t)$, we have 
\begin{equation}
  \label{eq:modorb}
  (\exp(th).F)(z) = \ee^{-t/2} F(\ee^{-t}z).
\end{equation}
From \eqref{eq:actonkz} we derive in particular
$U_h^{K_w}(t) = \ee^{t/2} K_{\ee^t w}$ for $t \in \R$, 
 hence $U_h^{K_w}(\ie t) = e^{\ie t/2} K_{\ee^{\ie t} w}.$ 
 For $w = \ie y$, $y > 0$, this specializes to
 $U_h^{K_{\ie y}}(\ie t) = \ee^{\ie t/2} K_{e^{\ie t} \ie y},$ 
   so that
   \[ \|   U_h^{K_{\ie y}}(\ie t)\|^2
     = K(\ee^{\ie t} \ie y, e^{\ie t} \ie y) 
     = \frac{1}{2 \Im(\ee^{\ie t} \ie y)}
     = \frac{1}{2 \cos(t) y}.\]
   By \eqref{eq:growthcond},
   this implies 
 \begin{equation}  
 	\label{Kiy_temp}
  K_{\ie y} \in \cH^J_{\rm temp}\text{ for every }y > 0.
  \end{equation}
   For these elements we find
   \[ \beta^+(K_{\ie y})(z) = U_h^{K_{\ie y}}(-\ie \pi/2)(z)
     = \ee^{-\ie \pi/4} K_y(z)
     = \ee^{-\ie \pi/4} \frac{\ie}{z - y}.\]
Using the isometric embeddings $H^2(\C_\pm)\hookrightarrow L^2(\R)$
via boundary values on the real line,
we obtain the orthogonal direct sum decomposition
\begin{equation}
	\label{left_decomp}
	L^2(\R)=H^2(\C_+)\oplus H^2(\C_-). 
\end{equation}
On the space $\cS'(\R)$ of tempered distributions, we consider 
the representation of $G$ defined by 
\begin{equation}
	\label{Uact}
	(U(b,a)F)(x) = a^{-1/2} F\Big(\frac{x+b}{a}\Big).
\end{equation}
It restricts to a unitary representation on $L^2(\R)$, 
for which the equation~\eqref{left_decomp} gives a decomposition into two irreducible subrepresentations~$U_\pm$. 
We obtain for the spaces of smooth vectors of these representations  
\begin{equation}
	\label{right_decomp_smooth}
	\cS(\R)\subseteq L^2(\R)^\infty=H^2(\C_+)^\infty\oplus H^2(\C_-)^\infty
\end{equation}
and, for the corresponding spaces of distribution vectors, the direct sum decomposition
\begin{equation}
	\label{right_decomp_distrib}
	H^2(\C_+)^{-\infty}\oplus H^2(\C_-)^{-\infty}=L^2(\R)^{-\infty}\subseteq \cS'(\R). 
\end{equation}
	One also has the representation of $G$ on $\cS'(\R)$ defined by 
\begin{equation}
	\label{piact} 
	(\widehat{U}(b,a)f)(p)  = \ee^{\ie bp}  a^{1/2} f(ap)
\end{equation}
that defines a unitary representation on $L^2(\R)$, 
for which the equation 
\begin{equation}
	\label{right_decomp}
	L^2(\R)=L^2(\R_+)\oplus L^2(\R_-)
\end{equation}
gives a decomposition into two irreducible subrepresentations~$\widehat{U}_\pm$. 

The following theorem is a key result of
this section. In Proposition~\ref{sharp_prop} it will be extended
to general antiunitary representations of $\Aff(\R)_e$. 

\begin{thm} \mlabel{thm:hardy} For the unitary representation
  $(U,\cH) = (U_+, H^2(\C_+))$ of $G = \Aff(\R)_e$,
  the subspace  $\cH^\omega(\Xi_1)$ is dense, but
\[ \cH^\omega(\Xi_1) \cap \cH^J_{\rm temp} = \{0\}.\]
\end{thm}

\begin{prf} 
	Since the right half plane $\C_r$ is dilation
    invariant and $(b,a) \mapsto a$ is a group homomorphism,
  $\Xi_1=\C\times \C_r\subseteq \Aff(\C)$ is $G$-biinvariant, i.e.,
  $G\Xi_1 G=\Xi_1$. 
Since $\cH^\omega(\Xi_1)\ne\{0\}$ by Proposition~\ref{prop:b.1}
in the appendix to this section and the irreducibility of $U_+$, 
the density of $\cH^\omega(\Xi_1)$ in $H^2(\C_+)$ follows 
from Lemma~\ref{invar_left_lem}\ref{invar_left_rem_item3}.

Let $F\in\cH^J\cap\cH^\omega(\Xi_1)$.
Then the orbit map $U^F$ extends to 
$\Xi_1 = \C \times \C_r$. 
The fact that it is an entire vector with respect to the translation group
$U_+\vert_{\R\times\{1\}}$ shows that $F$
extends analytically to~$\C$, 
and $JF = F$ means that
$F\res_{\ie \R}$ is real-valued.
For $z \in \C_+$ with $\Re z < 0$, we have for $0 < t < \pi/2$ 
\[ (\ee^{\ie t \partial U(h)}F)(z) = \ee^{-\ie t/2} F(\ee^{-\ie t}z),\]
so that the limit functions
\[ \xi^\pm := \lim_{t \to \mp\pi/2} \ee^{\ie t/2} \ee^{\ie t \partial U(h)}F
  \in \cH^{-\omega} \subeq \cO(\C_+) \]
satisfy 
\begin{equation}
  \label{eq:xi-rel}
  \xi^\pm(z) = F(\pm \ie  z) \quad \mbox{ for } \quad z \in \C_+, \ \pm\Re z > 0.
\end{equation}
We conclude that $\xi^\pm$ extend to entire functions 
and that \eqref{eq:xi-rel} holds for all $z\in\C$.
We also note that
\[ J \xi^+ = \lim_{t \to -\pi/2} \ee^{-\ie t/2} \ee^{-\ie t \partial U(h)} JF
  = \lim_{t \to \pi/2} \ee^{\ie t/2} \ee^{\ie t \partial U(h)} F = \xi^-.\] 

Assuming, in addition, that $F \in \cH^J_{\rm temp}$, we know 
from Proposition~\ref{prop:temp1}(b)
  that $\xi^\pm \in \cH^{-\infty}_{U_h}$, so that both
  have boundary values in $L^2(\R)^{-\infty}_{U_h}$ 
 (see~\eqref{right_decomp_distrib}), and this space
  consists of tempered distributions on $\R$ because
  $h$ acts by the Euler operator $x\frac{d}{dx}$, 
  hence $\cS(\R)\subseteq L^2(\R)^{\infty}_{U_h}$, cf. \eqref{right_decomp_smooth}--\eqref{right_decomp_distrib}. 
  Both distributions $\xi^\pm$ are boundary values of a holomorphic function
  on $\C_+$, extending to an entire function with real values on $\R$.
  Hence these distributions are real-valued.
  The relation $J \xi^+ = \xi^-$ and the action of $J$ on $L^2(\R)$
  by $(Jf)(x) = \oline{f(-x)}$
  now imply that $\xi^-(x) = (J\xi^+)(x) = \xi^+(-x)$
  (in the sense of distributions), but this implies that
  \[ \xi^\pm \in H^2(\C_+)^{-\infty} \cap  H^2(\C_-)^{-\infty}
    = \{0\}
    \]
where we use the direct sum decomposition~\eqref{right_decomp_distrib}.
\end{prf}

We now argue that, for the
  affine group $G = \Aff(\R)_e$, the crown domain
  $\Xi_1$ is too large but $\Xi_2$ is better. 

\begin{prop} 
\label{sharp_prop}
Let $G = \Aff(\R)_e$, $h = (0,1)$, 
$x = (1,0)$, and consider an antiunitary 
representation   $(U,\cH)$ of $G_{\tau_h}$.
\begin{itemize}
\item[\rm(a)] For 
$\Xi_1 = \C \times \C_r \subeq  \Aff(\C)$ 
we have $G\Xi_1 G=\Xi_1$, 
the linear subspace $\cH^\omega(\Xi_1)$ is $G$-invariant and dense in $\cH$, 
and we have 
\begin{equation}
\label{eq:goal1}
\oline{\cH^{\omega}(\Xi_1) \cap \cH^J_{\rm temp}} =\ker(\partial U(x)) \cap \cH^J.
\end{equation}
\item[\rm(b)] For
  $\Xi_2 = \{ (b,a) \in \Aff(\C) \: |\Im b| < \Re a\}$,
  the subspace $\cH^\omega(\Xi_2) \cap \cH^J_{\rm temp}$ is dense in
  $\cH^J$. 
\end{itemize}
\end{prop}

\begin{prf} (a) The subgroup $N:=\exp(\R x) \trile G$ is a normal $\tau_h$-invariant subgroup,  hence a normal subgroup of $G_{\tau_h}$.  
Using the notation of Lemma~\ref{lem:L1}\ref{lem:L1_item1}, we have 
$\cH^N=\ker(\partial U(x))$ and $\cH^N$ is $U(G_{\tau_h})$-invariant. 
Since $(U,\cH)$ is an antiunitary representation of $G_{\tau_h}$, 
the orthogonal complement  $(\cH^N)^\perp$ is $U(G_{\tau_h})$-invariant, too. 
By Lemmas \ref{lem:L2} and \ref{lem:L3} it suffices to prove the equality~\eqref{eq:goal1} for the subrepresentations of $U$ in the subspaces 
$\cH^N$ and $(\cH^N)^\perp$. 

As regards the subrepresentation on~$\cH^N$, 
we must check that, under the hypothesis in the statement, if additionally
$\partial U(x)=0$, then $\oline{\cH^{\omega}(\Xi_1) \cap \cH^J_{\rm temp}}
=\cH^J$. 
This is straightforward since $U$ factorizes through a
representation of the additive group $\R$,
hence its space of entire vectors is dense in~$\cH$. 
	
We now focus on the subrepresentations of $U$ in the subspace  $(\cH^N)^\perp$. 
In this case, what we actually have to do is to check that, under the hypothesis in the statement, if additionally $\ker (\partial U(x))=\{0\}$,
then $\cH^{\omega}(\Xi_1) \cap \cH^J_{\rm temp}=\{0\}$.
In view of \cite{GN47} or \cite[Prop.~2.38]{NO17},
there are Hilbert spaces $\cK_\pm$ such that the representation
$(U,\cH)$ is equivalent to the representation on 
\[  (\widehat{U}_+ \otimes \id_{\cK_+}) \oplus \ (\widehat{U}_-  \otimes \id_{\cK_-}).  \]
As $\widehat{U}_- \cong \widehat{U}_+^*$
 and 
$\widehat{U}_+$
 is equivalent to the
representation on the Hardy space $H^2(\C_+)$, the 
assertion follows from Theorem~\ref{thm:hardy}.

\nin (b) Since the assertion holds on $\cH^N$, which carries an
antiunitary representation of $\R^\times$, the argument under (a)
reduces the assertion to the
representation $U_+$ on the Hardy space $H^2(\C_+)$. 
We now recall from \eqref{Kiy_temp} that $K_{\ie y}\in\cH^J_{\rm temp}$ for all $y>0$. 
These functions span a dense subspace because
$0 = \la K_{\ie y}, F \ra = F(\ie y)$ for all $y > 0$ implies~$F = 0$.
For $(b,a) \in \Xi_2$, we have
$a\ie  - b \in \C_+$, and \eqref{eq:actonkz} yields
\[ U_+(b,a)K_{\ie}  = a^{1/2} K_{\oline a \ie  - \oline b},\]
showing that $K_\ie \in \cH^\omega(\Xi_2)$. Acting with
$U(0,a)$, $a > 0$, it follows that $K_{\ie y} \in \cH^\omega(\Xi_2) \cap
\cH^J_{\rm temp}$ for all $y > 0$. This completes the proof. 
 \end{prf}

\subsection{Hardy space representations of $\Aff_e(\R)$  and  $\SL_2(\R)$}
\label{subsec:versus}

We consider the unitary representation 
of the Lie group $\SL_2(\R)$ on the Hardy space~$H^2(\C_+)$ by 
\[	(\Pi_+(g)f)(z)
:=\frac{1}{cz+d}
\ 
f\Bigl(\frac{az+b}{cz+d}\Bigr)
\quad \mbox{ for } \quad
g^{-1} = \pmat{a & b \\ c &d} \in \SL_2(\R), \quad z \in \C_+, f\in H^2(\C_+).\] 
Here we use the left action of $\SL_2(\R)$ on $\C_+$ by M\"obius transformations
\begin{equation}
  \label{eq:moeb}
  \begin{pmatrix}a & b \\c & d	\end{pmatrix}. z=\frac{az+b}{cz+d}.
\end{equation}

\begin{rem} The relation with the representation 
	$(U_+,H^2(\C_+))$ 	of $\Aff_e(\R) = \R \rtimes \R_+$ by 
\begin{equation*}
(U_+(b,a)f)(z) = a^{-1/2} f\Big(\frac{z+b}{a}\Big)
\quad \mbox{ for } \quad z \in \C_+, f \in H^2(\C_+),
(b,a) \in \Aff(\R)_e \end{equation*}
satisfies $U_+=\Pi_+\circ\iota$ for the embedding 
\begin{equation}\label{iota:def}
\iota\colon\Aff_e(\R)\to\SL_2(\R),\quad 
\iota(b,a):=\begin{pmatrix}
	a^{1/2} & -ba^{-1/2} \\
	0 & a^{-1/2}
\end{pmatrix} 
\quad \mbox{ with } \quad
\Lie(\iota)(0,1)=\frac{1}{2}\begin{pmatrix}
			1 & 0 \\
			0 & -1
		\end{pmatrix} 
\end{equation}
Note that $\iota$ extends to an embedding
  $\Xi_1 = \C \times \C_r \into \SL_2(\C)$, given by the same formula,
  where $(e^z)^{1/2} = e^{z/2}$ for $|\Im z| < \pi/2$,
  denotes the canonical square root. 
\end{rem}

\begin{prop}\label{v1}
Consider the action of $\SL_2(\C)$ on $\C$ via~\eqref{eq:moeb} 
and the unitary irreducible representation $(U,\cH):=(\Pi_+,H^2(\C_+))$ of $\SL_2(\R)$,
\begin{itemize}
\item the Euler element $h:=\frac{1}{2}\begin{pmatrix}
1 & 0 \\ 0 & -1	\end{pmatrix}\in\fsl_2(\R)$, 
\item the conjugation $J\colon \cH\to\cH$, $Jf(z):=\oline{Jf(-\oline{z})}$,  
\end{itemize}
and the maximal compact subgroup $K:=\SO_2(\R)\subseteq\SL_2(\R)$.  
Define\begin{footnote}{Note that $v_1 = K_{\ie}$ corresponds to the evaluation functional in $\ie \in \C_+$ for the Hardy space.}\end{footnote}
\begin{equation*}
		v_1\colon\C_+\to\C,\quad v_1(z):=\frac{\ie}{z+\ie}\quad 
		\text{ and }\quad 
		\Xi_{\SL_2(\C)}:=\{g\in\SL_2(\C):g.(\pm\ie)\in\C_\pm\}.
	\end{equation*}
Then the following assertions hold: 
	\begin{enumerate}[{\rm(i)}]
		\item\label{v1_item1}
		$\Pi_+(k)v_1=v_1$ for every $k\in K$.   
              \item\label{v1_item2} $\Xi_{\SL_2(\C)}$ is open in
                $\SL_2(\C)$, $\SL_2(\R)\Xi_{\SL_2(\C)} \SO_2(\C)=\Xi_{\SL_2(\C)}$,
        and $v_1\in\cH^\omega_{\Pi_+}(\Xi_{\SL_2(\C)})$.  
		\item\label{v1_item3} $\exp(\cS_{\pm\pi/2} h)\subseteq\Xi_{\SL_2(\C)}$ and $v_1\in\cH^J_{\rm temp}$.  
\item\label{v1_item4} 
  $\Xi_2 = \iota_{\Xi_1}^{-1}(\Xi_{\SL_2(\C)})$ holds for the embedding
  $\iota_{\Xi_1} \: \Xi_1 \to \SL_2(\C), (b,a) \mapsto \pmat{a^{1/2} & -ba^{-1/2} \\ 0 & a^{-1/2}}$.
	\end{enumerate} 
\end{prop}

\begin{prf}
	\ref{v1_item1}
	Straightforward verification.
	
\nin	\ref{v1_item2}
	Since $\Xi_{\SL_2(\C)}\subseteq\SL_2(\C)$ is the pullback of the open subset $\C_+\subseteq\C$ through the orbit map $g\mapsto g.\ie$, it follows that $\Xi_{\SL_2(\C)}$ is open in $\SL_2(\C)$. 
	One has $\SL_2(\R)\Xi_{\SL_2(\C)} \SO_2(\C)\subseteq\Xi_{\SL_2(\C)}$ since the subgroup 
	$\SO_2(\C)$ leaves $\ie\in\C_+$ invariant, 
	while the subgroup $\SL_2(\R)$ leaves the upper half-plane $\C_+$ invariant. 
	Moreover, since $v_1\in\cH$ is $K$-fixed by \ref{v1_item1}, 
	while the representation $(U,\cH)$ is irreducible, 
	it follows by \cite[Thm. 6.6(iv)]{Kr09} that the orbit map
        $U^v(g) :=  U(g)v$ 
	extends to an analytic function $\Xi_{\SL_2(\C)}\to\cH$, that is, $v_1\in\cH^\omega_U(\Xi_{\SL_2(\C)})$ (cf.\ also Subsection~\ref{sec:hardy-affine-grp}).

\nin\ref{v1_item3} 
If $w=x+\ie y\in\C$, then $\exp(wh)=\diag(\ee^{w/2}, \ee^{-w/2})$,
hence
\begin{equation*}
\exp(wh).\ie=\ee^w\ie=\ee^x(-\sin y+\ie\cos y).
\end{equation*}
Thus, if $w\in\cS_{\pm \pi/2}$, then $y\in(-\pi/2,\pi/2)$ hence $\exp(wh).\ie\in\C_+$. On the other hand, since $v_1$ is a $K$-fixed vector by \ref{v1_item1}, 
we obtain $v_1\in\cH^J_{\rm temp}$ by \cite[\S 5.1, Cor. 5]{FNO25a}.
Alternatively, we have $v_1=K_{\ie}$, so that \eqref{Kiy_temp} applies. 
	
\nin 	\ref{v1_item4}: 
	is straightforward. 
\end{prf}

\begin{rem}
	The weight set of the representation $(\Pi_+,H^2(\C_+))$ of $\SL_2(\R)$ is semibounded. 
	(See for instance \cite[Thm. 6.4,p.~303]{Su90} for the analogous assertion on the weight set of the  Hardy space representation of $\SU_{1,1}(\C)$
        on the unit disk.)
	It then follows by \cite[Thm.~6.6(ii)]{Kr09} that the
        maximal domain of the analytic extension of the orbit map $U^v(g) = U(g)v$ 
	is actually bigger than $\Xi_{\SL_2(\C)}$.
        It consists of all $g \in \SL_2(\C)$ with $g.\ie \in \C_+$
     (see also Example~\ref{ex:affine-group}). 
\end{rem}

\subsection{Consequences for other Lie groups} 
  
Using the following lemma we will next derive
a consequence of Proposition~\ref{sharp_prop} that applies in particular to  solvable Lie algebras containing Euler elements.  
(See also Proposition~\ref{prop:E}.)  

\begin{lem}
\label{lem:abs-closed}
Any Lie group morphism $\varphi\colon\Aff(\R)_e \to G$ 
with discrete kernel is injective and its image
is a closed subgroup of the Lie group $G$. 
\end{lem}

\begin{prf} The hypothesis $\ker\Lie(\varphi)=\{0\}$ implies that $\ker\varphi$ is a discrete normal subgroup of $\Aff(\R)_e$, hence $\ker\varphi$ is contained in the center of the connected Lie group $\Aff(\R)_e$. 
Since the center of $\Aff(\R)_e$ is trivial, it follows that $\varphi$ is injective. Moreover, by 
\cite[Thm. 1.2(a)]{Om66}, the subgroup $\varphi(\Aff(\R)_e)$ is closed in $G$.
\end{prf}

\begin{prop}
\label{prop:22May2024_new}
Let $G$ be a 
Lie group with an inclusion $\eta \: G\to G_\C$ into a complex Lie
group~$G_\C$ with Lie algebra $\g_\C$. 
Assume that $h\in\g$ is an Euler element that gives rise to
an involutive automorphism $\tau_h\in\Aut(G)$, $(U,\cH)$ is an
antiunitary representation of the Lie group $G_{\tau_h}$,
and write $J:=U(\tau_h)$, and
\[ \g_{\pm 1}:=\{x\in\g:[h,x]=\pm x\}.\] 
Also assume that $N\subseteq G$ is a connected closed normal subgroup whose Lie algebra $\fn$ satisfies $\g=\fn\rtimes\R h$. 
For $\Xi:= N_\C \exp(\cS_{\pm \pi/2}h) \subseteq G_\C$,
we have: 
\begin{enumerate}[{\rm(a)}] \item\label{prop:22May2024_item1_new}
$\cH_{\rm temp}^J\cap \cH^\omega(\Xi)\subeq 
\bigcap_{x\in\g_1\cup\g_{-1}}\ker\partial U(x). $
\item\label{prop:22May2024_item2_new}
If $\g_{\pm 1}$ are not both contained in $\ker \dd U$,   
then  $\cH_{\rm temp}^J\cap \cH^\omega(\Xi)$ does not contain $G$-cyclic subspaces. 
\end{enumerate}
\end{prop}
	
\begin{prf}
\ref{prop:22May2024_item1_new}
For arbitrary 
$x\in\g_1\cup\g_{-1}$ we must prove that 
\begin{equation}
\label{sharp_prop_x}
\cH_{\rm temp}^J\cap \cH^\omega(\Xi)\subeq \ker\partial U(x).
\end{equation}
To this end we first note that 
$\g_1\cup \g_{-1}\subeq [\g,\g]$. 
On the other hand, since $\g=\fn\rtimes \R h$, we have $[\g,\g]\subeq \fn$, 
hence $\g_1\cup \g_{-1}\subeq\fn$. 
In particular, $x\in \fn$ and $[h,x]=\pm x$. 
It follows by Lemma~\ref{lem:abs-closed} that  $\exp(\R x+\R h)\subeq G$ 
is a simply connected closed subgroup. 
An application of Proposition~\ref{sharp_prop} 
for the representation $(U,\cH)$, restricted to the
subgroup $\exp(\R x+\R h)$, implies 
\begin{equation}
  \label{eq:p23}
  \cH_{\rm temp}^J\cap \cH^\omega(\Xi_x)\subeq \ker\partial U(x),
  \quad \mbox{ where } \quad
  \Xi_x:=\exp(\C x) \exp(\cS_{\pm \pi/2}h)\subseteq \Xi,
\end{equation}
hence $\cH^\omega(\Xi)\subeq\cH^\omega(\Xi_x)$, so that the
inclusion \eqref{eq:p23} implies~\eqref{sharp_prop_x}. 
		
\ref{prop:22May2024_item2_new}
Since $\g_\pm\subeq\g$ are abelian Lie subalgebras, it follows that $G_\pm:=\exp\g_\pm\subeq G$ are connected closed abelian subgroups. 
It is then straightforward that 
\[\bigcap_{x\in\g_\pm}\ker\partial U(x)=\cH^{G_\pm},
  \quad \mbox{ hence  } \quad
  \bigcap_{x\in\g_1\cup \g_{-1}}\ker\partial U(x)=\cH^{G_-}\cap \cH^{G_+}.\]
By Lemma~\ref{lem:L1}\ref{lem:L1_item2}, the right-hand side of the above equality is a closed $G$-invariant subspace,  hence does not contain $G$-cyclic subspaces unless $\cH^{G_-}= \cH^{G_+}=\cH$, which is equivalent
  to $\g_{\pm 1} \subeq \ker \dd U$.
Now the conclusion follows from~\ref{prop:22May2024_item1_new}. 
\end{prf}

\subsection{Appendix: Schober's example}
\mlabel{subsec:schober}

For an entire function \(F: \C \to \C\), \(b \in \C\) and \(a \in \C^\times\) we define the entire function
\[ \tilde U(b,a)F: \C \to \C, \quad z \mapsto F\left(\frac{z+b}a\right).\]

\begin{prop} \mlabel{prop:b.1} The space
\[\cH^\omega(\C \rtimes \C_r) := \{F \in \cO(\C): \left(\forall (b,a) \in \C \rtimes \C_r\right) \tilde U(b,a)F \in H^2(\C_+)\}\]
is non-zero.
\end{prop}

\begin{proof}
We consider the entire function 
\[ F \: \C \to \C, \quad F(z) :=  \int_0^\infty t^{-t} \ee^{\ie tz} \,dt.\]
By \cite[Eqn. (4)]{Ar07}
this function satisfies
\begin{equation}
	\label{eq:1}
  |F(x+\ie y)| \leq \frac 1{|x| - \frac \pi 2} \qquad
  \mbox{ for} \quad x,y \in \R, |x| > \frac \pi 2.
\end{equation}
Further \(F \neq 0\) since $F(0) = \int_0^\infty t^{-t} \,dt > 0.$ 
We will now show that \(F \in \cH^\omega(\C \rtimes \C_r)\). We first notice that, for \(x,y \in \R\) one has
\begin{equation}
		\label{eq:2}
  |F(x+\ie y)| = \left|\int_0^\infty t^{-t} \ee^{t\ie (x+\ie y)} \,dt\right| \leq \int_0^\infty \left|t^{-t} \ee^{t\ie (x+\ie y)}\right| \,dt = \int_0^\infty t^{-t} \ee^{-ty} \,dt
  =F(\ie y).
\end{equation}
Notice that, by the monotony of the exponential function, the function
\begin{equation*}
G\colon \R\to(0,\infty), \quad G(y):=F(\ie y)
\end{equation*} 
is decreasing. 

Now let \(b \in \C\) and \(a \in \C_r\). We write \(a = c+\ie d\) with \(c \in \R_+\) and \(d \in \R\) and set
\[\gamma := -\frac dc \qquad \text{and} \qquad \eta := \frac \pi 2 \frac {|a|^2}c.\]
Then, for every \(y \in \R\), setting \(\rho := y+\Im b\), one has
\begin{align}
&\int_\R \left|(\tilde U(b,a)F)(x+\ie y)\right|^2 \,dx 
= \int_\R \left|F\left(\frac{(x+\ie y)+b}{a}\right)\right|^2 \,dx \notag
  \\&= \int_\R \left|F\left(\frac{(x+\Re b)+\ie \rho}{a}\right)\right|^2 \,dx
  = \int_\R \left|F\left(\frac{x+\ie \rho}{a}\right)\right|^2 \,dx \notag
  \\
  	\label{eq:3bis}
  &= \int_\R \left|F\left(\frac{(x+\ie \rho)(c-id)}{|a|^2}\right)\right|^2 \,dx
  = \int_\R \left|F\left(\frac{xc+\rho d}{|a|^2} + \ie  \frac{\rho c-xd}{|a|^2}\right)\right|^2 \,dx.
\end{align}
Now, one has
\[\frac{xc+\rho d}{|a|^2} = \pm \frac \pi 2,
  \quad \mbox{ if and only if } \quad
  x = \frac 1c \left[-\rho d \pm \frac \pi 2 |a|^2\right] = \gamma \rho \pm \eta.\]
Using estimate~\eqref{eq:1} 
we then get
\begin{align*}
  &\int_{\R \setminus B_{2\eta}(\gamma \rho)} \left|F\left(\frac{xc+\rho d}{|a|^2} + \ie  \frac{\rho c-xd}{|a|^2}\right)\right|^2 \,dx
  \leq \int_{\R \setminus B_{2\eta}(\gamma \rho)} \left|\frac 1{\left|\frac{xc+\rho d}{|a|^2}\right|-\frac \pi 2}\right|^2 \,dx
\\&=\frac{|a|^4}{c^2} \int_{\R \setminus B_{2\eta}(\gamma \rho)} \left|\frac 1{\left|x-\gamma \rho\right|-\eta}\right|^2 \,dx
=\frac{2|a|^4}{c^2} \int_\eta^\infty \frac 1{x^2} \,dx
=\frac{2|a|^4}{c^2} \cdot \frac 1\eta
= \frac 4 \pi \frac{|a|^2}c.
\end{align*}
On the other hand, using estimate~\eqref{eq:2}
and the monotony of the above function $G(y) = F(\ie y)$, 
we have
\begin{align*}
  &\int_{B_{2\eta}(\gamma \rho)} \left|F\left(\frac{xc+\rho d}{|a|^2} + \ie  \frac{\rho c-xd}{|a|^2}\right)\right|^2 \,dx
    \leq \int_{B_{2\eta}(\gamma \rho)} \left|G\left(\frac{\rho c-xd}{|a|^2}\right)\right|^2 \,dx  \\
  &\leq \int_{B_{2\eta}(\gamma \rho)} \left|G\left(\frac{\rho c-(\gamma \rho+2\eta)d}{|a|^2}\right)\right|^2 \,dx
  =4\eta \left|G\left(\frac{\rho c-(\gamma \rho+2\eta)d}{|a|^2}\right)\right|^2\\
  &=4\eta \left|G\left(\frac{\rho (c-\gamma d)}{|a|^2} - \frac{2\eta d}{|a|^2}\right)\right|^2
    =4\eta \left|G\left(\frac{\rho}{c} - \frac{\pi d}{c}\right)\right|^2
=\frac{2\pi|a|^2}c \left|G\left(\frac{y+\Im b-\pi d}{c}\right)\right|^2.
\end{align*}
Now, by equation~\eqref{eq:3bis}, 
we have
\begin{align*}
&\int_\R \left|(\tilde U(b,a)F)(x+\ie y)\right|^2 \,dx \\
&\qquad= \int_\R \left|F\left(\frac{xc+\rho d}{|a|^2} + \ie  \frac{\rho c-xd}{|a|^2}\right)\right|^2 \,dx
\\&\qquad= \int_{\R \setminus B_{2\eta}(\gamma \rho)} \left|F\left(\frac{xc+\rho d}{|a|^2} + \ie  \frac{\rho c-xd}{|a|^2}\right)\right|^2 \,dx + \int_{B_{2\eta}(\gamma \rho)} \left|F\left(\frac{xc+\rho d}{|a|^2} + \ie  \frac{\rho c-xd}{|a|^2}\right)\right|^2 \,dx
\\&\qquad\leq \frac 4 \pi \frac{|a|^2}c + \frac{2\pi|a|^2}c \left|G\left(\frac{y+\Im b-\pi d}{c}\right)\right|^2.
\end{align*}
This, using the monotony of the function \(G\), yields
\begin{align*}
\left\lVert \tilde U(b,a)F\right\rVert_{H^2}^2 &= \sup_{y > 0} \int_\R \left|(\tilde U(b,a)F)(x+\ie y)\right|^2 \,dx
\\&\leq \sup_{y > 0} \left[\frac 4 \pi \frac{|a|^2}c + \frac{2\pi|a|^2}c \left|G\left(\frac{y+\Im b-\pi d}{c}\right)\right|^2\right]
\\&= \frac 4 \pi \frac{|a|^2}c + \frac{2\pi|a|^2}c \left|G\left(\frac{\Im b-\pi d}{c}\right)\right|^2 < \infty
\end{align*}
and therefore \(\tilde U(b,a)F \in H^2(\C_+)\).
\end{proof} 

The estimates in the preceding proof imply that the function
\[ \C \times \C_r \to H^2(\C_+), \quad (b,a) \mapsto
  \tilde U^F(b,a):= \tilde U(b,a) F \]
  is locally bounded. 
  Since it is pointwise holomorphic, it follows
  from \cite[Cor.~A.III.3]{Ne00} that it is holomorphic.
  As $F\res_{\ie \R}$ is real, we also have $JF = F$.
  For the boundary values of $\tilde U(0,a)F$ for $a = -\ie$, we
  obtain the function $F(\ie z)$ in $\cO(\C_+)$. The
  boundary values of this function on $\R$ are given by the real-valued
  function
  \[ H(x) := \int_{\R_+} p^{-p} \ee^{px}\, dp \]
  which grows for $x \to \infty$ faster than any function $\ee^{px}$.
  Therefore $H$ is not a tempered distribution, hence in particular
  not contained in the space $H^2(\C_+)^{-\infty}_{U_h}$.

\section{Regularity of representations and existence of nets}
\mlabel{sec:4}

In this section we show that,
for an antiunitary representation $(U,\cH)$ of $G_{\tau_h}$,  
$h$-regularity is equivalent to the existence of a net 
  of real subspaces satisfying (Iso), (Cov), (RS) and (BW)
  for some open subset $W \subeq G$ (Theorem~\ref{thm:reg-net}). 

\begin{thm} \mlabel{thm:reg-net}
    Let $(U,\cH)$ be an antiunitary representation of
  $G_{\tau_h}$ and
  $\sV = \sV(h,U) \subeq \cH$ the corresponding standard subspace.
  Then there exists a net $(\sH(\cO))_{\cO \subeq G}$ on
  open subsets of $G$, satisfying
  {\rm(Iso), (Cov), (RS)} and {\rm(BW)} for some
  open subset $W \subeq G$, if and only if $U$ is $h$-regular, i.e.,
  $\sV_N := \bigcap_{g \in N} U(g)\sV$ is cyclic for some $e$-neighborhood $N \subeq G$.
  
  Specifically, this holds for all $e$-neighborhoods for which
  there exists an open subset $\cO \subeq W$ with
  $N\cO \subeq W$, i.e., $N \subeq \bigcap_{g \in \cO} Wg^{-1}$. 
\end{thm}

\begin{prf} ``$\Rarrow$'': If a net $\sH$ with the asserted properties
  exists, then $\sV = \sH(W)$ by (BW), and for any relatively compact open subset $\cO
  \subeq W$ there exists an $e$-neighborhood $N \subeq G$ with 
  $N\cO \subeq W$. Then, for all $g^{-1} \in N$, we have
  \[ U(g)^{-1}\sH(\cO) = \sH(g^{-1}.\cO)
    \subeq \sH(W)= \sV,\quad \mbox{ hence } \quad
    \sH(\cO) \subeq \sV_N.\]
  Now (RS) implies that $\sH(\cO)$ is cyclic, so that $\sV_N$ is also
  cyclic, and thus   $U$ is $h$-regular.

  \nin ``$\Larrow$'':  Assume that $\sV_N$ is cyclic for an
  $e$-neighborhood $N$. Pick an open $e$-neighborhood $N_1 \subeq N$
  with $N_1 N_1^{-1} \subeq N$. Then
  \[ W := \exp(\R h) N_1 \]
  is an open subset of $G$. We define the net
  $\sH$ on open subsets of $G$ by
  \[ \sH(\cO) = \bigcap_{g \in G, \cO \subeq gW} U(g)\sV.\]
  This net satisfies (Iso) and (Cov) by \cite[Lemma~2.17]{MN24}.

  We now verify the Reeh--Schlieder property (RS).
  So let $\eset\not=\cO \subeq G$ be an open subset.
  By (Iso) and (Cov), it suffices to show that $\sH(\cO)$ is cyclic
  if $\cO \subeq N_1$. Then, for $g \in G$,
  the inclusion $\cO \subeq gW = g \exp(\R h) N_1$ implies
  \[ g \in \cO N_1^{-1} \exp(\R h) \subeq N_1 N_1^{-1} \exp(\R h)
    \subeq N \exp(\R h),\]
  so that
  \[ \sH(\cO)
    \supeq \bigcap_{g \in N\exp(\R h)} U(g) \sV 
    = \bigcap_{g \in N} U(g) \sV = \sV_N \]
  implies that $\sH(\cO)$ is cyclic. This proves (RS).
  It follows in particular that $\sH(W)$ 
  is cyclic, so that \cite[Lemma~2.17(c)]{MN24} implies 
  $\sH(W)=~\sV$. Therefore (BW) is also satisfied.
\end{prf}

\begin{rem} Note that $v \in \sH(\cO)$ is equivalent to
\[ g^{-1}\cO  \subeq W \quad \Rarrow \quad U(g)^{-1} v \in \sV.\] 
  If $\cO \subeq W$ is relatively compact, this condition
  holds for $g$ in an $e$-neighborhood. Therefore $\sH(\cO)$
  consists of vectors $v \in \cH$ whose orbit map
  $U^v \: G \to \cH$ maps an $e$-neighborhood into~$\sV$
  (cf.\ Proposition~\ref{prop:4.8}(d)). Put differently,
  the subset $(U^v)^{-1}(\sV) \subeq G$ has interior points.
\end{rem}

\begin{rem} (a)  Suppose that $v \in \sV \cap \cH^\omega$ is an analytic vector
  and $U(N)v \subeq \sV$ holds for an $e$-neighborhood $N\subeq G$. 
  Then uniqueness of analytic continuation implies 
  $U(G)v \subeq \sV$, i.e., $v \in \sV_G$.

  If, in addition, $v$ is $G$-cyclic, then $\sV_G$ is a cyclic real subspace,
  so that its invariance under the modular group of $\sV$ implies with
  \cite[Prop.~3.10]{Lo08}  that $\sV = \sV_G$, i.e., that
  $\sV$ is $G$-invariant. 
Then $U(G)$ commutes with the Tomita operator of $V$, 
  hence with both $J_\sV$ and the modular group
  $\Delta_{\sV}^{\ie \R} = e^{\R \partial  U(h)}$. 
In particular, $[h,\g]\subeq\ker\dd U$ 
  hence $U\circ \tau_h=U$.  
  Since $J_\sV$ commutes with $U(G)$, 
  it also follows that 
$\cH^{J_\sV}$ is a real orthogonal
  representation of $G$, and $U$ is the complexification,
  considered as a representation of $G_{\tau_h}$ on $\cH \cong (\cH^{J_\sV})_\C$.
  This is the context where $\partial U(h)$ and $J_\sV$ commute with the whole
  group~$G_{\tau_h}$.

  \nin (b) Another perspective on (a) is that for the cyclic subrepresentation
  generated by any $v \in \cH^\omega \cap \sV_N$ 
  the operators
  $\partial U(h)$ and $J_\sV$ commute with $U(G)$.
  So $v$ is fixed by the normal subgroup $B$ 
  with Lie algebra
  \[ \fb := \g_1 + [\g_1,\g_{-1}] + \g_{-1}.\]
\end{rem}

\section{Regularity for the split oscillator group}
\mlabel{sec:5}

The smallest pair $(G,h)$ for which all  the $h$-regularity
  criteria from \cite{MN24} fail, is the split 
  oscillator group  of dimension~$4$.
  In this section we show in particular that, also for this group,
  all antiunitary representations are $h'$-regular for
  all Euler elements $h' \in \g$.

The split oscillator group is
\[ G := \Heis(\R^2) \rtimes_\sigma \R
  \quad \mbox{ with } \quad
  \sigma_t = \ee^{th}, \quad h = \pmat{1 & 0 \\ 0 & -1}
  \quad \mbox{ and }\quad  \g = \heis(\R^2) \rtimes \R h,\]
where $\Heis(\R^2)$ is the $3$-dimensional Heisenberg group. 
Note that $h$ is an Euler element in $\g$. 
We choose a basis $p,q,z \in \heis(\R^2)$
with 
\begin{equation}
\label{eq:osc-def}
 [q,p] = z, \quad [h,q] = q, \quad [h,p] = -p, \quad [h,z] = 0.
\end{equation}
The corresponding involution is given by
\[ \tau_h^\g(z,q,p,t) = (z,-q,-p,t).\] 

We realize $\g$ as a subalgebra of $\g^* := \fsl_3(\R)$, where
\[ h = \frac{1}{3}
  \pmat{1 & 0 & 0 \\ 0 & -2 & 0 \\ 0& 0 & 1}, \quad
q = \pmat{0 & 1 & 0 \\ 0 & 0 & 0 \\ 0& 0 & 0}, \quad
p = \pmat{0 & 0 & 0 \\ 0 & 0 & 1 \\ 0& 0 & 0}, \quad
z = \pmat{0 & 0 & 1 \\ 0 & 0 & 0 \\ 0& 0 & 0}.\]
Here a key point is that $h$ is an Euler element of $\fsl_3(\R)$, 
i.e., $V = \R^3$ is a $2$-graded $\g$-module:
\[ V = V_{1/3} \oplus V_{-2/3}, \quad
 V_{1/3} = \R \be_1 + \R \be_3, \quad 
 V_{-2/3} = \R \be_2.\] 

The involution $\tau_h$ on $\SL_3(\R)$ corresponds to conjugation
with
\[ \kappa := \exp(\pi \ie h) = \diag(\zeta, \zeta^{-2}, \zeta), \quad
  \zeta = \ee^{\pi \ie /3}.\]
It satisfies $\kappa^6 =\bone$ and $\kappa^3 = - \bone$, so that
$\Ad(\kappa) = \tau_h^\g$ is an involution. 
Note that $G = N \rtimes \exp(\R h)$, where
$N$ is the unipotent radical of a minimal parabolic
subgroup
\begin{equation}
  \label{eq:g*}
  G\subeq P = M AN \subeq G^* := \SL_3(\R),
\end{equation}
where $M$ is
  the subgroup of diagonal matrices with entries $\{\pm 1\}$,
  $A$ the subgroup of positive diagonal matrices,
and $N$ the unipotent upper triangular matrices. 

\begin{lem} \mlabel{lem:twist-osci}
  Let $G = N \rtimes \R$ be a semidirect product
  and $h = (0,1) \in \g$ an Euler element.
  For a unitary representation $(U,\cH)$ of $G$ and $\lambda \in \R$, we
  consider the twisted representation
  \[ U_\lambda(n,t) := e^{i\lambda t} U(n,t).\]
  Then $U$ is $h$-regular in the sense of
  {\rm Definition~\ref{def:unit-regular}} 
if and only if all twists $U_\lambda, \lambda \in \R,$
  are $h$-regular.   
\end{lem}

\begin{prf} We show that, if $\Omega\subeq G$ is a subset containing $e$,
  then the subset 
\[ \cD_\Omega = \{ \xi \in \cH \:
  (\forall g \in \Omega)\quad U(g)^{-1}\xi \in \cD(\Delta^{1/2})
  \ \mbox{ and }  \ 
   \Delta^{1/2} U(g)^{-1} \xi = 
   U(\tau_h(g))^{-1}  \Delta^{1/2} \xi\} \]
 is the same for all representations $U_\lambda$.
 We put
 \[ \Delta_\lambda := e^{2\pi \ie \partial U_\lambda(h)}
   = e^{2\pi \ie \lambda} \Delta.\]
 Then $\cD(\Delta_\lambda^{1/2}) = \cD(\Delta^{1/2})$ for each
 $\lambda$, so that the condition
 $U_\lambda(g)^{-1}\xi \in \cD(\Delta_\lambda^{1/2})$ is
 independent of~$\lambda$. Likewise
 \[    \Delta^{1/2}_\lambda U_\lambda(n,t)^{-1} \xi
   = e^{\pi \ie \lambda} \Delta^{1/2} U(n,t)^{-1} e^{-i\lambda t} \xi 
   =   U_\lambda(\tau_h(n,t))^{-1}  \Delta_\lambda^{1/2} \xi \]
holds for all $\lambda$ if it holds for $\lambda =0$.
\end{prf}

\begin{thm} \mlabel{thm:split-osc-reg}
  All unitary representations of the split oscillator group
  $G = N \rtimes_\alpha \R$ are $h$-regular.
\end{thm}

\begin{prf} Consider a spherical principal series representation 
  of $G^* = \SL_3(\R)$ on $\cH = L^2(K/M) \cong L^2(N)$.
Then $N$ acts by the left regular
representation and $A$ acts by automorphisms on $N$.
This representation of $\SL_3(\R)$ is $h$-regular by \cite[Cor. 4.24]{MN24},
so that its restriction to $G$ is $h$-regular as well. 
As a representation of $G$, it
decomposes as a direct integral of Schr\"odinger representations 
$U_{\chi}$ for parameters $\chi \in \R^\times
\subeq \R \cong \hat{Z(N)}$, where we use that
the action of $\exp(\R h)$ on $A$ by automorphisms can be implemented
in a canonical way by the metaplectic representation of
$\tilde\Sp_2(\R)$,
in the Schr\"odinger representation of $N$ on~$L^2(\R)$. 
So, by \cite[Lemma C.1]{MN24}, almost all these representations are
$h$-regular. As they are conjugate under outer automorphisms induced
by~$A$, all these representations are $h$-regular.
By metaplectic factorization (\cite[\S IX.4]{Ne00}),
it follows that every irreducible representation of $G$ whose restriction to
$Z(N)$ is non-trivial is isomorphic to a twist
$U_{\chi,\lambda}$, $\lambda \in \R$, in the sense of Lemma~\ref{lem:twist-osci}.
We conclude that  all these representations are $h$-regular as well.

Irreducible representations that are trivial 
 on $Z(N)$ are representations of the quotient group
 $\R^2 \rtimes \R$ whose representations are $h$-regular
 by \cite[Prop.~4.15]{MN24}.
So all irreducible representations of $G$
 are $h$-regular. Now the $h$-regularity of all unitary
 representations of $G$ follows from the existence of a 
  direct integral decomposition  and~\cite[Lemma~4.4]{MN24}.
\end{prf}

\subsubsection*{Euler elements of the split oscillator Lie algebra}

The following proposition shows that the set $\cE(\g)$
of all  Euler elements of the split oscillator Lie algebra $\g$ is the disjoint union of two parallel affine hyperplanes in~$\g$. 
This should be compared with the classification of Euler elements in simple Lie algebras in \cite[Subsect. 3.3]{MN21}.

\begin{prop}
	\label{prop:split-E}
For the split oscillator Lie algebra $\g$, the set $\cE(\g)$
is the disjoint union of 
the hyperplanes 
	\[  \pm h + [\g,\g] = \pm h + \heis(\R^2).\] 
\end{prop}

\begin{prf} Consider the quotient Lie algebra
        $\fq := \g/\fz \cong \R^2 \rtimes \so_{1,1}(\R)$
        (the Poincar\'e algebra  in dimension~$2$) and write
        $h_\fq$ for the image of $h$ in $\fq$. Then the adjoint
        orbit $\cO^\fq_{h_\fq}$ of $h_\fq$ satisfies 
\[ \cO^\fq_{h_\fq} = h_\fq + [h_\fq, \fq] = h_\fq +  (\R^2 \oplus \{0\}). \]
Therefore all elements in $\pm h + \heis(\R^2)$ are
        Euler elements because their image in $\fq$ is conjugate to
        $h_\fq$ or $-h_\fq$. That these are all non-central Euler elements
        follows from the description of Euler elements in $\fq$
        and $\pm h + \fz \subeq \cE(\g)$.
\end{prf}

The preceding proposition is in strong contrast
with the semisimple case, where the set of Euler elements
has a more complicated, curved structure.

\begin{rem} Combining the preceding argument with the observation
  that the additive group $\Hom(\g/[\g,\g],\fz)$ acts by
  automorphisms  of the form $\alpha.x := x + \alpha(x)$ on $\g$
  implies that the group $\Aut(\g)$ acts transitively on the
  set $\cE(\g)$ of all 
  Euler elements in $\g$.

  On the group level, this implies that Theorem~\ref{thm:split-osc-reg}
  also holds for all other 
  Euler elements
  $h' \in \cE(\g)$. 
\end{rem}

\begin{rem} If $h_1 = h + z$ with $z \in \fz(\g)$, then
  $\tau_{h_1}^\g = \tau_h^\g$ implies that the group $G_{\tau_h}$
  does not depend on the choice of the Euler element $h$.
  For the associated modular operator we then have
  \[ \Delta_1 = \Delta e^{2\pi \ie \partial U(z)},\]
  with $J e^{i \partial U(z)} J = e^{-i \partial U(z)}$, so that
  \[ \sV_1 := \Fix(J (\Delta_1)^{1/2})
    = \Fix(e^{-\frac{\pi \ie}{2} \partial U(z)} J
    \Delta e^{\frac{\pi \ie}{2} \partial U(z)})
    = e^{-\frac{\pi \ie}{2} \partial U(z)}\sV,\]
    and, for every $e$-neighborhood $N \subeq G$, we have
$ \sV_{1,N} = e^{-\frac{\pi \ie}{2} \partial U(z)} \sV_N.$ 
      This shows that $U$ is $h$-regular if and only if it is
      $h'$-regular.   
\end{rem}

\begin{rem}
We can also use the realization of $G$ in $G^* := \SL_3(\R)$ to
find a crown domain for~$G$. To this end, we consider
the non-compactly causal symmetric space
\[ M^* := G^*/H^* = \{ g B g^\top \: g \in G^*\}, \quad
  B := \diag(-1,1,-1).\]
It consists of all symmetric matrices $A \in \Sym_3(\R)$ with
$\det(A) = 1$ and signature $(1,2)$.
Accordingly,
\[  M^*_\C \cong \{ A \in \Sym_3(\C) \: \det(A) = 1\}.\]

As the stabilizer algebra
\[ \fh^*  =  \left\{ \pmat{ 0 & a & b \\ a & 0 & c \\ - b & c & 0}
  \: a,b,c \in \R \right\} \cong \so_{1,2}(\R) \]
intersects $\g$ (and even $\fp$) trivially, 
the orbit of the base point yields an open embedding
$G \into M^*$. 
\end{rem}

\section{Perspectives and open problems}

An interesting open problem in the context of this paper
is the regularity conjecture:

\begin{conj} If $G$ is a simply connected Lie group and $h \in \g$ an
  Euler elements, then every antiunitary representation
  of $G_{\tau_h}$ is $h$-regular, i.e., there exists an $e$-neighborhood $N \subeq G$ for which $\sV_N := \bigcap_{g \in N} U(g)\sV(h,U)$ is cyclic.
\end{conj}

In \cite{MN24} one finds various sufficient conditions
for representations to be $h$-regular and the split oscillator group
was the smallest example for which none of these conditions were satisfied.
A~natural next step would be to verify the conjecture for all solvable
Lie groups.

\begin{prob} Suppose that, for an antiunitary representation
  $(U,\cH)$ of $G_{\tau_h}$, a net $\sH$ satisfying (Iso), (Cov), (RS) and
  (BW) exists. Does this imply the existence of a triple 
  $(G, \Xi, h)$ satisfying (Cr1-3) such that
  $\cH^\omega(\Xi) \cap \cH^J_{\rm temp}$ is $G$-cyclic? 
\end{prob}

\begin{prob} Suppose that, for an antiunitary representation
$(U,\cH)$ of $G_{\tau_h}$, we have a 
cyclic vector $v \in \cH^\omega \cap  \cH^J_{\rm temp}$.
Then, for some sufficiently small $\Xi_0$ of the form
$\Xi_0 = G \exp(\ie\Omega_0)$, $\Omega_0 \subeq \g$ an open convex
$0$-neighborhood, we have
$v \in \cH^\omega(\Xi_0)$. Is it possible to enlarge
$\Xi_0$ to a crown $\Xi$ of $G$ with $v \in \cH^\omega(\Xi)$?

This a question about ``minimal Riemann domains over $G_\C$''
to which $G$-orbit maps extend.
We would like to extend to
$\exp(\cS_{\pm \pi/2}h) \Xi_0^{\oline\tau_h}$ by
\[ U^v(\exp(zh) p) = e^{z \partial U(h)} U^v(p). \]
This should be a vector with the property that, for
$w \in \cH^{\omega}_{U_h}(\cS_{\pm \pi/2})$, it satisfies
\[ \la w, U^v(\exp(zh) p) \ra
  = \la e^{-\oline z \partial U(h)}w,  U^v(p) \ra 
  = \la U^w(-\oline z),  U^v(p) \ra . \]
This suggests to consider orbit maps with values in
the dual space of the dense Fr\'echet subspace
$\cH^{\omega}_{U_h}(\cS_{\pm \pi/2}) \subeq \cH$. 
\end{prob}

\begin{prob}
To study locality phenomena in AQFT, i.e., pairs of subspaces
$\sH(\cO_1)$ and $\sH(\cO_2)$ with 
$\sH(\cO_1) \subeq \sH(\cO_2)'$, which is the key property to
relate our abstract constructions to Physics, one
needs Euler elements $h \in \g$ which are symmetric in the sense that
$-h \in \Ad(G)h$. Then $h \in [\g,\g]$, and this never
happens if $\g$ is solvable. Therefore one always needs non-trivial
Levi complements containing symmetric Euler elements 
(\cite[Prop.~3.2]{MN21}).
In \cite{MNO25} and \cite{NO25}, first steps are made for semisimple
groups, but it also remains to be explored for which ``mixed groups'',
such as $G = R \rtimes \SL_2(\R)$  with a solvable group $R$,
similar structures exist.
An important mixed example is the Poincar\'e group
$G = \R^{1,d} \rtimes \SO_{1,d}(\R)_e$. 
\end{prob}

\appendix

\section{Spaces of smooth and analytic vectors}
\mlabel{app:c} 

\subsection{The space of analytic vectors and its dual} 

In this subsection we briefly discuss the space of analytic
vectors of a unitary representation of a Lie group.
Let $(U,\cH)$ be a unitary representation of
the connected real Lie group $G$. We write
\[ \cH^\omega = \cH^\omega_U\subeq \cH \]
for the space of {\it analytic vectors}, i.e.,
those $\xi\in \cH$ for which the orbit map
$U^\xi \colon G \to \cH, g \mapsto U(g)\xi$,  is analytic.

To endow $\cH^\omega$ with a locally convex topology,
we specify subspaces $\cH^\omega_{V}$ by open convex  {$0$-neighborhoods}
$V \subeq \g$ as follows. Let $\eta_G \: G \to G_\C$ be the
  universal complexification of $G$ and recall our assumption that
  $\ker(\eta_G)$ is discrete.
We assume that $V$ is so small that the map
\begin{equation}
  \label{eq:eta-g-v}
 \eta_{G,V} \colon G_V := G \times V
 \to G_\C, \quad (g,x) \mapsto \eta_G(g) \exp(\ie x)
\end{equation}
is a covering. Then we endow $G_V$ with the unique complex manifold
structure for which $\eta_{G,V}$ is holomorphic.
We now write $\cH^\omega_V$ for the set of those analytic vectors
$\xi$ for which the orbit map $U^\xi\colon G \to \cH$ extends to a
holomorphic map
\[ U^\xi_V \colon G_V\to \cH.\]
As any such extension is $G$-equivariant by 
uniqueness of analytic continuation, it must have the form
\begin{equation}
  \label{eq:orb-map}
 U^\xi_V(g,x) = U(g) \ee^{\ie \partial U(x)} \xi \quad \mbox{ for }  \quad
 g \in G, x \in V,
\end{equation}
so that
$\cH^\omega_V \subeq \bigcap_{x \in V} \cD(\ee^{\ie \partial U(x)}).$ 
By \cite[Lemma~1]{FNO25a}, we actually  have equality.

We topologize the space
$\cH^\omega_V$ by identifying it with 
$\cO(G_V, \cH)^G$, the Fr\'echet space of $G$-equivariant holomorphic maps
$F \colon G_V \to \cH$, endowed with the Fr\'echet topology of
  uniform convergence on compact subsets. Now
$\cH^\omega = \bigcup_V \cH^\omega_V$, 
and we topologize $\cH^\omega$ as the locally convex direct limit
of the Fr\'echet spaces $\cH^\omega_V$.
If  $\eta_G \colon  G \to  G_\C$
  is injective, it is easy to see that
  we thus obtain the same topology as in \cite{GKS11}. Note that,
for any monotone
basis  $(V_n)_{n \in \N}$ of convex $0$-neighborhoods in $\g$, we
then have
\[ \cH^\omega \cong \indlim \cH^\omega_{V_n},\]
so that $\cH^\omega$ is a countable locally convex limit of
Fr\'echet spaces. As the evaluation maps
\[ \cO(G_V, \cH)^G \to \cH, \quad F \mapsto F(e,0) \] 
are continuous, the inclusion
$\iota \colon \cH^\omega \to \cH$ 
is continuous. 

We write $\cH^{-\omega}$ for the space  of continuous antilinear functionals 
$\eta\colon \cH^\omega \to \C$ (called {\it hyperfunction vectors})
and
\[ \la \cdot, \cdot \ra \colon \cH^\omega \times \cH^{-\omega} \to \C \]
for the natural sesquilinear pairing that is linear in the second argument.
We endow $\cH^{-\omega}$ with the weak-$*$ topology. 
We then have natural continuous inclusions
\[ \cH^\omega \into \cH \into \cH^{-\omega}.\] 

Our specification of the topology on $\cH^\omega$
differs from the one 
\cite{GKS11} because we do not want to assume that
the $\eta_G \colon G \to G_\C$ is injective, 
but both constructions define the same topology.
Moreover,  the arguments in \cite{GKS11} apply with minor changes to
general Lie groups.

\subsection{The space of smooth vectors and its dual}

The subspace  
$\cH^\infty \subeq \cH$ of vectors $v \in \cH$ for which the orbit map 
$U^v \colon G \to \cH, g \mapsto U(g)v,$  is smooth 
({\it smooth vectors}) is dense
and carries a natural Fr\'echet topology for which the action of 
$G$ on this space is smooth (\cite{Ne10}).
We have a representation of $\fg$ on $\cH^\infty$ given by
$\dd U(x)v= \frac{d}{dt}\big|_{t=0}U(\exp tx)v$.
The topology is defined by the seminorms $\| \dd U(x)u\|,
x \in \cU(\g)$. The space of {\it distribution vectors},
denoted by $\cH^{-\infty}$, is the conjugate linear dual of $\cH^\infty$.
The elements of $\cH^{-\infty}$ are called
{\it distribution vectors}. It contains in particular the functionals 
$\la \cdot, v \ra$, $v \in \cH$.
Here we follow the convention common in physics to
require inner products to be conjugate linear in the first and 
complex linear in the second argument. 
We thus obtain complex linear embeddings 
\begin{equation}
  \label{eq:incl4}
\cH^{\omega} \subeq \cH^\infty \subeq \cH \subeq \cH^{-\infty}
\subeq \cH^{-\omega},
\end{equation}
where all inclusions are continuous
and $G$ acts on all spaces 
by representations denoted, 
\[ U^\omega, \quad U^\infty, \quad U, \quad U^{-\infty} \quad
  \mbox{ and } \quad U^{-\omega},\] respectively.

All of the three above  representations can be integrated to the
convolution algebra $C^\infty_c(G) := C^\infty_c(G,\C)$ of
test functions, for instance
\[ U^{-\infty}(\phi) := \int_G \phi(g)U^{-\infty}(g)\, dg,\] 
where $dg$ stands for a left Haar measure on $G$.

\section{Regularity for unitary representations}

Let $(U,\cH)$ be a unitary representation of $G$ and
$\oline\cH$ the complex Hilbert space obtained from $\cH$
by flipping the complex structure. On the space
$\tilde\cH := \cH \oplus \oline\cH$ 
we consider the antiunitary representation $(\tilde U, \tilde \cH)$,
defined by
\[ \tilde U\res_G := U \oplus (\oline U \circ \tau_h)
  \quad \mbox{ and }  \quad  \tilde J(v,w) = (w,v),\]
where $\oline U$ is the representation obtained from $U$ by
flipping the complex structure on~$\cH$.
Then
\[  \tilde \Delta := e^{2\pi \ie \partial \tilde U(h)}
  = \Delta \oplus \Delta^{-1}
   \quad \mbox{ for } \quad
    \Delta := e^{2\pi \ie \partial U(h)},  \]
and the corresponding standard subspace is given by
\[ \tilde \sV := \sV(h,\tilde U) = \Gamma(\Delta^{1/2})
  = \{ (v, \Delta^{1/2}v) \: v \in \cD(\Delta^{1/2})\}
= \Fix(\tilde J \tilde\Delta^{1/2}) \]
(cf.~\cite[Lemma~2.22]{MN24}).

For $g_0 \in G$ and $\eset \not=\Omega \subeq G$,
the subspace $\tilde \sV_{g_0\Omega} = U(g_0)\tilde\sV_\Omega$
is cyclic if and only if $\sV_\Omega$ is cyclic.
We therefore {\bf assume from now on
  that $e \in \Omega$.} We consider the subspace
\[  \cD_\Omega
  := \{ \xi \in \cD(\Delta^{1/2}) \:
  (\forall g \in \Omega)\ \tilde U(g)^{-1}(\xi,\Delta^{1/2}\xi) \in
  \Gamma(\Delta^{1/2})\}
  = \{  \xi \in \cD(\Delta^{1/2}) \: (\xi, \Delta^{1/2}\xi) \in \tilde \sV_\Omega\}\]
for which 
\[ \tilde\sV_\Omega = \{ (\xi, \Delta^{1/2} \xi) \: \xi \in \cD_\Omega\}.\]  
We observe that
\begin{equation}
  \label{eq:darel}
  \cD_\Omega = \{ \xi \in \cH \:
  (\forall g \in \Omega)\quad U(g)^{-1}\xi \in \cD(\Delta^{1/2})
  \ \mbox{ and }  \ 
   \Delta^{1/2} U(g)^{-1} \xi = 
 U(\tau_h(g))^{-1}  \Delta^{1/2} \xi\}, 
\end{equation}
which only involves the unitary representation $U$. 

\begin{rem} \mlabel{rem:bgl02}
For an antiunitary representation $(U,\cH)$ of $G_{\tau_h}$, 
the method of proof of 
\cite[Prop.~4.1/2]{BGL02} implies 
\begin{equation}
 \mlabel{eq:bgl02}
\cD_\Omega = \sV_\Omega + \ie \sV_\Omega,
\end{equation}
so that the subspace $\cD_\Omega$ is dense in $\cH$ if and only if
$\sV_\Omega$ is cyclic. As $\Gamma(\Delta^{1/2})$ is closed,
the definition of $\cD_\Omega$ also implies that
\[ \cD_{\oline \Omega} = \cD_\Omega \quad \mbox{ for } \quad e\in \Omega \subeq G.\] 
\end{rem}

The main advantage of the subspace $\cD_\Omega$ is that it only requires
the unitary representation $U\res_G$, so that $\cD_\Omega$ can be
defined for any unitary representation $(U,\cH)$ of $G$.

\begin{lem} \mlabel{lem:3.28} For $e \in \Omega \subeq G$ and a unitary
  representation $(U,\cH)$ of $G$, the following are equivalent:
  \begin{itemize}
  \item[\rm(a)] $\cD_\Omega$ and $\Delta^{1/2} \cD_\Omega$ are both dense in $\cH$. 
  \item[\rm(b)] For $\tilde \sV = \sV(h,\tilde U)$
    the real subspace $\tilde \sV_\Omega$ is cyclic in $\tilde \cH$.
  \end{itemize}
\end{lem}

The assumptions of this lemma are satisfied if
    $\cD_\Omega$ is a core for $\Delta^{1/2}$.

\begin{prf} Let $I(v,w) = (\ie v,-\ie w)$ denote the complex structure on
  $\tilde\cH = \cH \oplus \oline\cH$.
  First we claim that
  \begin{equation}
    \label{eq:tildeveq}
    \tilde \sV_\Omega + I \tilde \sV_\Omega = \cD_\Omega \oplus \Delta^{1/2} \cD_\Omega.
  \end{equation}
  The lemma follows directly from this identity.
  
  The definition of $\cD_\Omega$ implies that
$\tilde \sV_\Omega \subeq \cD_\Omega \oplus \Delta^{1/2} \cD_\Omega.$ 
For $\xi \in \cD_\Omega$, we further observe that 
\[ I(\xi, \Delta^{1/2}\xi)
  = (\ie\xi, -\ie\Delta^{1/2}\xi)
  = (\ie\xi, \Delta^{1/2}(-\ie\xi))
  \in \cD_\Omega \oplus \Delta^{1/2}\cD_\Omega.\]
This proves ``$\subeq$'' in \eqref{eq:tildeveq}. 
For the converse inclusion, let
$v, w \in \cD_\Omega$. Then
\[ (v, \Delta^{1/2}w)
  = \Big(\frac{v + w}{2}, \Delta^{1/2}\Big(\frac{v+w}{2}\Big)\Big)
+  \Big(\frac{v - w}{2}, -\Delta^{1/2}\Big(\frac{v-w}{2}\Big)\Big)
\in \tilde \sV_\Omega + I \tilde \sV_\Omega. \]
This proves equality in \eqref{eq:tildeveq}. 
\end{prf}

\begin{prop} \mlabel{prop:3.29}  
  If $(U,\cH)$ is an antiunitary representation of $G_{\tau_h}$
  and $J = U(\tau_h)$, then
  $\cD_\Omega = \sV_\Omega + \ie \sV_\Omega$
  is  dense in $\cH$ if and only if $\tilde\sV_\Omega$ is cyclic in $\tilde\cH$. 
\end{prop}

\begin{prf}
  If $(U,\cH)$ is an antiunitary representation of $G_{\tau_h}$
  and $J = U(\tau_h)$, then
\[ J \sV_\Omega
  = \bigcap_{g \in \Omega} U(\tau_h(g)) J \sV 
  = \bigcap_{g \in \Omega} U(\tau_h(g)) \sV'
  = \sV_{\tau^G_h(\Omega)}'.\]
Therefore $\sV_\Omega$ is cyclic if and only if
$\sV'_{\tau^G_h(\Omega)}$ is cyclic.
By 
\eqref{eq:bgl02} in Remark~\ref{rem:bgl02},
the first of these conditions
is equivalent to the density of $\cD_\Omega$ in $\cH$, 
and the
second to the density of $\cD_{\tau^G_h(\Omega)}'$ in $\cH$, where
$\cD'_\Omega$ is the domain defined by exchanging $h$ and $-h$, resp.,
$\Delta$ and $\Delta^{-1}$. 
Concretely, we have
\begin{align*}
 \cD_{\tau^G_h(\Omega)}'
&= \{ \xi \in \cD(\Delta^{-1/2}) \:
  (\forall g \in \Omega)\ \tilde U(\tau_h(g))^{-1}(\xi,\Delta^{-1/2}\xi) \in
                      \Gamma(\Delta^{-1/2})\} \\
  &= \{ \xi \in \cD(\Delta^{-1/2}) \:
    (\forall g \in \Omega)\ \Delta^{-1/2}U(\tau_h(g))^{-1}\xi
    = U(g)^{-1} \Delta^{-1/2}\xi \} \\   
  &= \Delta^{1/2}\{ \eta \in \cD(\Delta^{1/2}) \: 
    (\forall g \in \Omega)\ \Delta^{-1/2}U(\tau_h(g))^{-1}\Delta^{1/2}\eta 
    = U(g)^{-1} \eta \} \\   
  &= \Delta^{1/2}\{ \eta \in \cD(\Delta^{1/2}) \: 
    (\forall g \in \Omega)\ U(\tau_h(g))^{-1}\Delta^{1/2}\eta 
    = \Delta^{1/2} U(g)^{-1} \eta \}\\
  &= \Delta^{1/2} \cD_\Omega.
\end{align*}
This shows that, for an antiunitary representation
of $G_{\tau_h}$, the density of $\Delta^{1/2} \cD_\Omega$ follows from
the density of $\cD_\Omega$.
Hence the assertion follows from Lemma~\ref{lem:3.28}.
\end{prf}

\begin{defn} \mlabel{def:unit-regular} A unitary
  representation $(U,\cH)$ of $G$ is called {\it $h$-regular}
  if there exists an $e$-neighborhood $N \subeq G$ such that,
  for $\Delta = e^{2\pi \ie \partial U(h)}$, 
  the subspaces
  $\cD_N$ from \eqref{eq:darel} and $\Delta^{1/2} \cD_N$
  are both dense in $\cH$ (cf.\ Lemma~\ref{lem:3.28}).   
\end{defn}

\begin{rem} (a) By Lemma~\ref{lem:3.28}, a unitary representation
  $(U,\cH)$ is $h$-regular if and only if
  its natural antiunitary extension $(\tilde U, \tilde \cH)$ is $h$-regular
  in the sense that, for some $e$-neighborhood $N \subeq G$, the subspace
  $\tilde \sV_N$ is cyclic (\cite{MN24}).

\nin   (b)  Proposition~\ref{prop:3.29} shows that,
  if $(U,\cH)$ extends to an antiunitary representation
  of $G_{\tau_h}$ on $\cH$, then $\tilde\sV_N$ is cyclic in $\tilde\cH$
  if and only if  $\sV_N$ is cyclic in $\cH$. Therefore regularity
  of unitary representations in
  the sense of Definition~\ref{def:unit-regular}
  is equivalent to the regularity concept
  for antiunitary representations from~\cite[Def.~4.1]{MN24}.   
\end{rem}

The following lemma, transcribed from \cite[Lemma~4.25]{Ne25}
concerning antiunitary representations, is a useful tool.

  \begin{lem} \mlabel{lem:g-intera}
    For a unitary  representation $(U,\cH)$ of $G$,
    the  following assertions hold:
  \begin{itemize}
  \item[\rm(a)] If $U = U_1 \oplus U_2$ is a direct sum, then
    $U$ is $h$-regular if and only if $U_1$ and $U_2$ are $h$-regular.
  \item[\rm(b)] If $U$ is $h$-regular, then every subrepresentation is
    $h$-regular. 
  \item[\rm(c)]  Assume that $G$ has at most countably many connected
    components and let  $U = \int_X^\oplus U_m \, d\mu(m)$ be a unitary
    direct integral  representation of $G$. Then $U$  is regular
    if and only if there exists an $e$-neighborhood $N \subeq G$
    such that, for $\mu$-almost every $m \in X$,
    the subspaces $\cD_N$ and $\Delta^{1/2} \cD_N$ are dense. 
      \end{itemize}
  \end{lem}

\section{Generalities on Euler elements and representations}
\label{subsect:E}

The following proposition shows that every solvable Lie algebra with a non-central Euler element is the semidirect product of a 3-graded solvable Lie algebra via the action of $\R$ defined by its grading derivation. 

\begin{prop}
	\label{prop:E}
	Let $\g$ be a finite-dimensional solvable real
        Lie algebra with center~$\fz(\g)$. 
	If $h\in\g\setminus\fz(\g)$ is an Euler element, then there exists an ideal $\fn\subeq\g$ with $[\g,\g]\subeq\fn$ and $\g=\fn\rtimes\R h$. 
\end{prop}

\begin{prf}
	The Lie algebra $\g$ is solvable, hence 
	$[\g,\g]$ is a nilpotent ideal of $\g$ 
	by \cite[Ch. 1, \S 5, no. 3, Cor. 5]{Bo98}, 
	hence $[\g,\g]$ is contained in the greatest nilpotent ideal of~$\g$. 
	Then 
	for every $x\in[\g,\g]$ the map $\ad x\colon\g\to\g$ is nilpotent 
	by \cite[Ch. 1, \S 5, no. 3, Cor. 7]{Bo98}. 
	Since $h\in\g\setminus\fz(\g)$ is an Euler element, we thus see that $h\not\in[\g,\g]$. 
	This implies that there exists a linear subspace $\fn\subseteq\g$ satisfying 
	$[\g,\g]\subeq\fn$, $\dim(\g/\fn)=1$, and $h\not\in\fn$. 
	In particular, $[\g,\fn]\subeq[\g,\g]\subeq\fn$, 
	hence  $\fn\subeq\g$  is a  hyperplane ideal and $\g=\fn\rtimes\R h$. 
\end{prf}

Lemmas \ref{lem:L1}, \ref{lem:L2}, and \ref{lem:L3} below are needed in the proof of Propositions \ref{sharp_prop} and \ref{prop:22May2024_new}. 
We recall the notation $N_G(H):=\{g\in G: gHg^{-1}= H\}$  for every group $G$ with a subgroup $H\subseteq G$. 
Also, we denote by $\la X\ra\subeq G$ the subgroup generated by a
subset $X\subeq G$. 

\begin{lem}
	\label{lem:L1}
	Let $(U,\cH)$ be an antiunitary representation of a group $G$. 
	\begin{enumerate}[{\rm(i)}]
		\item\label{lem:L1_item1}
		If $N\subeq G$ is a normal subgroup,  then the fixed-point set 
		$\cH^N:=\{v\in\cH : (\forall n\in N)\ U(n)v=v\}$  
		is a closed $G$-invariant subspace of~$\cH$. 
		\item\label{lem:L1_item2}
		If $G_0,G_\pm\subeq G$ are subgroups satisfying $G=\la G_-\cup G_0\cup G_+\ra$ and $G_0\subeq N_G(G_\pm)$, then $\cH^{G_-}\cap\cH^{G_+}$ is a closed $G$-invariant subspace of~$\cH$. 
	\end{enumerate}
\end{lem}

\begin{prf}
	\ref{lem:L1_item1}
        For $g \in G$, $n \in N$ and $v \in \cH^N$, we have
        $U(n) U(g) v= U(g) U(g^{-1}ng)v = U(g) v,$
        so that $U(G)\cH^N\subseteq\cH^N$.
	
	\ref{lem:L1_item2}
	The subspace $\cH_0:=\cH^{G_-}\cap\cH^{G_+}$ is clearly $G_\pm$-invariant. 
	Then, since $G=\la G_-\cup G_0\cup G_+\ra$, it suffices to show that $\cH_0$ is $G_0$-invariant. 
	
	Since $G_0\subeq N_G(G_\pm)$, it follows that $G^\pm:=G_\pm G_0$ is a subgroup of $G$ and $G_\pm$ is a normal subgroup of $G^\pm$. 
	Then, by Assertion~\ref{lem:L1_item1}, the subspace $\cH^{G_\pm}$ is $G^\pm$-invariant. 
	Since $G_0\subseteq G^\pm$, it follows that  $\cH^{G_\pm}$ is $G_0$-invariant, 
	hence $\cH_0$ is $G_0$-invariant, and this completes the proof. 
\end{prf}

We omit the obvious proofs of the following two lemmas.

\begin{lem}
	\label{lem:L2}
	Let $(U,\cH)$ be a unitary representation of a connected Lie group $G$ 
	and consider an orthogonal direct sum decomposition
 $\cH=\cH_1\oplus\cH_2$ into $G$-invariant subspaces. 
Then, for every crowned Lie group $(G,\Xi,h)$, 
        we have $\cH_j^\omega(\Xi)=\cH_j\cap\cH^\omega(\Xi)$ for $j=1,2$,
        and the direct sum decomposition 
	$\cH^\omega(\Xi)=\cH_1^\omega(\Xi)\oplus\cH_2^\omega(\Xi)$. 
\end{lem}

\begin{lem}
	\label{lem:L3}
        Let $G$ be a connected Lie group 
          and $h\in\g$ an Euler element for which $G_{\tau_h}$
          exists. 
          Consider anantiunitary representation $(U,\cH)$ of $G_{\tau_h}$
          and an orthogonal direct sum decomposition into
          $G_{\tau_h}$-invariant subspaces $\cH=\cH_1\oplus\cH_2$. 
For $J:=U(\tau_h)$ and $J_k:=J\vert_{\cH_k}$ for $k=1,2,$ 
we then have $\cH^{J_k}_{k,\ {\rm temp}}=\cH_k\cap \cH^J_{\rm temp}$ for $k=1,2$,  
	and the direct sum decomposition 
        \[ \cH^J_{\rm temp}=\cH^{J_1}_{1,\ {\rm temp}}\oplus
          \cH^{J_2}_{2,\ {\rm temp}}. \]
\end{lem}


\begin{thebibliography}{aaaaaaaa}



\bibitem[AG90]{AG90}
    Akhiezer, D. N., and S. G. Gindikin, {\it On Stein 
      extensions of real symmetric spaces}, Math. Ann. {\bf 286} (1990), 1--12
  

\bibitem[Ar07]{Ar07} Armitage, D.~H., {\it 
    Entire functions that tend to zero on every line},
  Amer. Math. Monthly {\bf 114:3} (2007), 251--256 
  
\bibitem[BN24]{BN24} Belti\c t\u a, D., and K.-H.\ Neeb,
  {\it Holomorphic extension of one-parameter operator groups},
  Pure Appl. Funct. Anal. {\bf 9:6} (2024), 1483--1526;  arXiv:2304.09597

 
\bibitem[Bo98]{Bo98}
Bourbaki, N., ``Lie Groups and Lie Algebras,'' Chapters  1–3, 
  	Elements of Mathematics, 
  	Springer-Verlag, 
  	Berlin, 1998 


\bibitem[BR96]{BR96}
Bratteli, O., and D.~W.~Robinson, ``Operator Algebras and Quantum Statistical
Mechanics II,'' 2nd ed.,
Texts and Monographs in Physics, Springer-Verlag, 1996 

\bibitem[BGL02]{BGL02} Brunetti, R., Guido, D., and R. Longo, {\it 
Modular localization and Wigner particles}, Rev. Math. Phys. {\bf  14} (2002), 759--785


\bibitem[FN{\'O}25a]{FNO25a} Frahm, J., K.-H. Neeb, and G.~\'Olafsson,
{\it Nets of standard subspaces on 
 non-compactly causal   symmetric spaces}, in 
``Symmetry in Geometry and Analysis,'' Volume 2, 115--195;
Birkh\"auser/Springer, Singapore, 2025


\bibitem[FN{\'O}25b]{FNO25b} Frahm, J., K.-H. Neeb, and G.~\'Olafsson,
{\it Realization of unitary representations of the  Lorentz group
  on de Sitter space},
Indag. Math. {\bf 36:1} (2025), 61--113 


\bibitem[GN47]{GN47} Gelfand, I., and M.~Naimark, {\it Unitary
representations of the group of linear transformations of the
straight line}, Dolk. Akad. Nauk. SSSR {\bf 55} (1947), 567--570
        
\bibitem[GKS11]{GKS11} Gimplerlein, H., B. Kr\"otz, and H. Schlichtkrull,
  {\it Analytic representation theory of Lie groups: general theory
   and analytic globalization of Harish--Chandra modules},
  Compos. Math. {\bf 147:5} (2011), 1581--1607;
  corrigendum ibid. {\bf 153:1} (2017), 214--217;
   arXiv:1002.4345v2 


\bibitem[HC53]{HC53}
Harish-Chandra, 
{\it Representations of a semisimple Lie group on a Banach space. I}. 
Trans. Amer. Math. Soc. {\bf 75} (1953), 185--243.


\bibitem[HiNe12]{HiNe12}
Hilgert, J., and K.-H.~Neeb, 
``Structure and Geometry of Lie Groups,'' 
Springer Monographs in Mathematics. Springer, New York, 2012


\bibitem[Kr08]{Kr08} Kr\"otz, B., 
{\it Domains of holomorphy for irreducible unitary representations of simple Lie groups}, Invent. Math. {\bf 172:2} (2008), 277--288

\bibitem[Kr09]{Kr09}
Kr\"otz, B., 
{\it Crown theory for the upper half plane}. 
In: Ginzburg, D., E. Lapid, and D. Soudry (eds.), 
``Automorphic Forms and $L$-functions I. Global Aspects'', 
Contemp. Math. {\bf 488}, Israel Math. Conf. Proc., 
Amer. Math. Soc., Providence, RI, 2009, pp. 147--182.


\bibitem[KSt04]{KSt04} Kr\"otz, B and R. J. Stanton, {\it Holomorphic extensions of representations. I. Automorphic functions},
   Ann. of Math. (2) {\bf 159} (2004),  641--724

\bibitem[LL15]{LL15} Lechner G. and R.~Longo, {\it
    Localization in nets of standard spaces},
  Comm. Math. Phys. {\bf 336} (2015), 27--61
  

 \bibitem[Lo08]{Lo08}  Longo, R., {\it Real Hilbert subspaces, modular theory, 
$\SL(2, \R)$ and CFT} 
in ``Von Neumann Algebras in Sibiu'', 33-91, Theta Ser. Adv. Math. {\bf 10}, 
Theta, Bucharest, 2008  

\bibitem[Lo69]{Lo69} Loos, O., ``Symmetric Spaces I: General Theory,'' 
W. A. Benjamin, Inc., New York, Amsterdam, 1969


\bibitem[MMTS21]{MMTS21} Morinelli, V., Morsella, G., Stottmeister, A., and
  Y.~Tanimoto, {\it Scaling limits of lattice quantum fields by wavelets},
  Commun. in Math. Phys. {\bf 387:1} (2021), 299--360

\bibitem[MN21]{MN21} Morinelli, V., and K.-H. Neeb, 
{\it Covariant homogeneous nets of standard subspaces}, 
Comm. Math. Phys. {\bf 386} (2021), 305--358  


\bibitem[MN24]{MN24} Morinelli, V., and K.-H. Neeb, 
{\it From local nets to Euler elements}, 
Adv. Math. {\bf 458} (2024), part A, Paper No. 109960, 87 pp. 


\bibitem[MNO24]{MNO24} Morinelli, V., K.-H. Neeb, and G.\ \'Olafsson, 
  {\it Modular geodesics and wedge domains
    in general non-compactly causal symmetric spaces}, 
Annals of Global Analysis and Geometry
{\bf 65:1} (2024), Paper No. 9, 50pp 

\bibitem[MN\'O25]{MNO25} Morinelli, V., K.-H. Neeb, and G. \'Olafsson, {\it 
    Orthogonal pairs of Euler elements}, 
  in preparation

\bibitem[Ne00]{Ne00} Neeb, K.-H., ``Holomorphy and Convexity in Lie Theory,'' 
Expositions in Mathematics {\bf 28}, de Gruyter Verlag, Berlin, 2000 

\bibitem[Ne10]{Ne10} Neeb, K.-H.,  {\it 
On differentiable vectors for representations of infinite dimensional 
Lie groups}, J. Funct. Anal. {\bf 259} (2010), 2814--2855

\bibitem[Ne22]{Ne22} Neeb, K.-H., 
{\it Semigroups in 3-graded Lie groups and endomorphisms of standard 
  subspaces}, Kyoto Math. Journal {\bf 62:3} (2022), 577--613;
arXiv:OA:1912.13367 

\bibitem[Ne25]{Ne25} Neeb, K.-H., ``Nets of Real Subspaces on Homogeneous
  Spaces and Algebraic Quantum Field Theory,''
  IHP Lecture Notes, March 2025;
\url{www.math.fau.de/wp-content/uploads/sites/3/2025/03/qft-lect.pdf}


\bibitem[N\'O17]{NO17} Neeb, K.-H., and G.\, \'Olafsson,  {\it 
Antiunitary representations and modular theory}, 
in ``50th Sophus Lie Seminar'', Eds. K. Grabowska et al, 
Banach Center Publications {\bf 113}; pp.~291--362; 
arXiv:math-RT:1704.01336 


\bibitem[N\'O21]{NO21} Neeb, K.-H., and G.\, \'Olafsson, 
{\it Nets of standard subspaces on Lie groups}, 
Advances in Math. {\bf 384} (2021), 107715, arXiv:2006.09832

\bibitem[N\'O23]{NO23} Neeb, K.-H., and G.\, \'Olafsson, 
  {\it Wedge domains in non-compactly causal symmetric spaces},
 Geometriae Dedicata {\bf 217:2} (2023), Paper No. 30;
 arXiv:2205.07685
 
\bibitem[N\'O25]{NO25} Neeb, K.-H., and G.\, \'Olafsson, 
  {\it Locality on non-compactly causal symmetric spaces},
  in preparation

\bibitem[N\'O\O21]{NOO21} Neeb, K.-H., G.\, \'Olafsson, and B. \O{}rsted, 
{\it Standard subspaces of Hilbert spaces of holomorphic 
functions on tube domains}, Comm. Math. Phys. 
{\bf 386} (2021), 1437--1487; arXiv:2007.14797 


\bibitem[Om66]{Om66}
Omori, H., 
{\it Homomorphic images of Lie groups}. 
J. Math. Soc. Japan {\bf 18} (1966), 97--117


\bibitem[Si24]{Si24} Simon, T., {\it
    Polynomial growth of holomorphic extensions of orbit maps of K-finite vectors at the boundary of the crown},
  Preprint, arXiv:2403.13572
  
\bibitem[Su90]{Su90} Sugiura, M., 
``Unitary Representations and Harmonic Analysis. An Introduction''. 
Second edition. North-Holland Mathematical Library {\bf 44}. North-Holland Publishing Co., Amsterdam; Kodansha, Ltd., Tokyo, 1990 


\bibitem[Wa72]{Wa72}
Warner, G., ``Harmonic Analysis on Semi-simple Lie Groups. I'',  
Grundlehren der math. Wissenschaften {\bf 188}, 
Springer-Verlag, New York-Heidelberg, 1972 
  

\end{thebibliography}
\end{document}